\documentclass[11pt]{article}

\usepackage{amsmath, amsthm, amsfonts}
\usepackage{listings}
\usepackage{amsfonts}
\usepackage{amssymb}
\usepackage{float}
\usepackage{anysize}
\usepackage{graphicx}
\usepackage[usenames]{color}

\usepackage{latexsym}
\usepackage{enumerate}
\usepackage{layout}
\usepackage{dsfont}
\usepackage{eufrak}

\newtheorem{thm}{Theorem}[section]

\newtheorem{lem}[thm]{Lemma}

\theoremstyle{definition}           

\theoremstyle{remark}
\newtheorem{rem}[thm]{Remark}
\newtheorem{teo}[thm]{Theorem}

\title{Numerical scheme for stochastic differential equations driven by fractional Brownian motion with $ 1/4<H <1/2$.}
\author{ H\'ector Araya$^{1}$ \ \ \ \ Jorge A. Le\'on$^{2}$ \ \ \ \ Soledad Torres$^{3}$  \\
\small $^{1}$ CIMFAV, Facultad de Ingenier\'ia, Universidad de Valpara\'iso,\\
Casilla 123-V, 4059 Valpara\'iso, Chile.\\
  hector.araya@postgrado.uv.cl\\ 
\small $^{2}$ Depto. de Control Autom\'atico, Cinvestav-IPN, \\ 
Apartado Postal 14-740, Ciudad de M\'exico, 07000, M\'exico. \\
jleon@ctrl.cinvestav.mx \\
  \small $^{3}$CIMFAV, Facultad de Ingenier\'ia, Universidad de Valpara\'iso,\\
Casilla 123-V, 4059 Valpara\'iso, Chile.\\
soledad.torres@uv.cl\\
}

\begin{document}
\maketitle

\abstract{In this article, we study  a numerical scheme for stochastic differential equations driven by fractional Brownian motion with Hurst parameter $ H \in \left( 1/4, 1/2 \right)$. Towards this end, we apply Doss-Sussmann representation of the solution and an approximation of this representation using a first order Taylor expansion. The obtained rate of convergence is  $n^{-2H +\rho}$, for $\rho$ small enough.}



\section{Introduction}

In this article we are interested in a pathwise approximation of the solution to the stochastic differential equation 
\begin{equation}\label{eqest}
X_{t} = x + \int_{0}^{t} b(X_{s})ds + \int_{0}^{t} \sigma(X_{s}) \circ dB_{s},  \quad  t \in [0,T], 
\end{equation}
where $x \in \mathbb{R}$ and $b, \sigma : \mathbb{R} \rightarrow \mathbb{R} $ are measurable functions. The stochastic integral in (\ref{eqest}) is understood in the sense of Stratonovich, (see  Al{\`o}s et.al. \cite{alos1} for details) and $B= \lbrace  B_{t} , t \in [0,T] \rbrace$ is a fractional Brownian motion (fBm)  with Hurst parameter $H \in (1/4,1/2)$. $B$ is a centered Gaussian process with a covariance structure given by
\begin{equation}\label{covariance}
\mathbb{E}\left( B_{t} B_{s} \right) = {1 \over 2}\left(  t^{2H} + s^{2H} - \vert t-s \vert^{2H} \right), \quad  t \in [0,T]. 
\end{equation}

 In \cite{alos1}, the existence and uniqueness for the solution of equation  (\ref{eqest}) have been established under suitable conditions, which follows  from our assumption (see hypothesis (H) in Section \ref{doss-sussmann}).

Equation (\ref{eqest}) has been analyzed by several authors, for different interpretations of stochastic integrals, because of the properties 
of fractional Brownian motion $B$  . Among these properties, we can mention
self-similarity, stationary increments, $\rho$-H\"older continuity, for any
$\rho\in(0,H)$, and the covariance of its increments on intervals decays
 asymptotically as a negative power of the distance between the intervals.
Therefore, equation (\ref{eqest}) 
 becomes quite  useful in applications in different areas  such as physics, biology, finance, etc (see, e.g., \cite{Alos2007, kaj2008, KLUP}). Hence, it is important to provide approximations to the solution of (\ref{eqest}).

For  $H=1/2$ (i.e., B is a Brownian motion), a large number of numerical schemes to approximate the unique solution of (\ref{eqest}) has been considered in the literature. The reader can consult  Kloeden and Platen  \cite{klo} (and the references therein), for a complete  exposition of this topic. In particular, Talay  \cite{Talay} introduces the Doss-Sussmann transformation  \cite{doss, sussmann} in the study of numerical methods to the solution of stochastic differential equations (see Section \ref{doss-sussmann} for the definition of this transformation).

For $H>1/2$, numerical schemes for equation (\ref{eqest}) have been analyzed by several authors. For instance, we can mention    \cite{araya, hu, mish1}  and \cite{nourdin}, where the stochastic integrals is interpreted as the extension of the Young integral given in \cite{za} and the forward integral, respectively. It is well-known that these integrals agree with the Stratonovich one under suitable conditions (see  Al{\`o}s and Nualart \cite{alos2}).  
 
In this paper we are interested in the case $H<1/2$, because numerical schemes for the solution to (\ref{eqest})  have been studied only in some particular situations. Namely, Garz\'on et. al. \cite{garzon}  use the Doss-Sussmann transformation in order to prove the convergence for the Euler scheme associated to (\ref{eqest}) by means of an approximation of fBm via fractional transport processes. In \cite{nourdin}, the authors also take advantage of the Doss-Sussmann transformation in order to discuss the Crank-Nicholson method, for $ H \in \left(1/6 , 1/2\right)$ and $b\equiv0$. Here,   they show convergence in law of the error to a random variable, which depends on the solution of the equation and  an independent Gaussian random variable. Specifically, the authors state that the rate of convergence of the scheme is of order $n^{1/2-3H}$. In \cite{tindel1} the authors consider the so-called modified Euler scheme for multidimensional stochastic differential equations  driven by fBm with $ H \in \left(1/3 , 1/2\right)$. They utilize rough paths techniques in order to obtain the convergence rate of order $n^{1/2-2H}$. Also, they prove that this rate is sharp. In \cite{deya}  a numerical scheme for  stochastic differential equations  driven by a multidimensional fBm with Hurst parameter greater than $1/3$ is introduced. The method is based on a second-order Taylor expansion, where the L\'evy area terms are replaced by products of increments of the driving fBm. Here, the order of convergence is $n^{-(H-\rho)}$, with $\rho \in \left( 1/3 , H \right)$. In order to get this rate of convergence, the authors use  a combination of rough paths techniques and error bounds for the discretization of the L\'evy area terms.

In this work we propose an approximation scheme  for the solution to (\ref{eqest}) with $H \in \left( 1/4 , 1/2 \right)$. To do so, we use a first order Taylor expansion in the Doss-Sussmann representation of the solution. We consider the case $H \in \left( 1/4 , 1/2 \right)$ because it is showed in
 \cite{alos2}  that the solution of (\ref{eqest}) is given by this transformation. However,  even in the case $\left( 0 , 1/4 \right)$, our scheme tends to the mentioned transformation. The rate of convergence in this paper is $ n^{- 2H  + \rho}$, where $\rho <  2H$ small enough, improving the ones given in  \cite{nourdin}, \cite{Talay}, \cite{deya} and \cite{tindel1}. Also  our rate is better than the one obtained in  \cite{garzon} when the fBm is not approximated  by means of fractional transport process.  
We observe that our method only establishes this rate of convergence for $H<1/2$ because we could only see that the
auxiliary  inequality (\ref{dify1}) below  is satisfied in this case. However, the same construction holds for $H>1/2$ (see \cite{nourdin},  Proposition 1). In this case, the rate of convergence for the scheme is not the same as  the case  $1/4 < H < 1/2$. In fact, for $H>1/2$, we only get that 
 the rate of convergence  is $n^{-1 + \rho}$ for $\rho$ small enough.

The  paper is organized as follows: In Section \ref{sec2} we introduce the notations needed in this article. In particular, we explain the Doss-Sussmann-type transformation related to the unique solution to (\ref{eqest}). Also, in this section, the scheme is presented and the main result is stated (Theorem \ref{teo1} below). In Section \ref{sec3}, we establish the auxiliary lemmas, which are needed to show, in Section \ref{sec4}, that the main result is true. The proof of the auxiliary lemmas  are presented in Section \ref{sec5}. Finally, in the Appendix  (Section  \ref{apen}), other auxiliary result is also studied because it is a general result concerning  the Taylor expansion for some continuous functions.

\section{Preliminaries and main result} \label{sec2}
In this  section, we  introduce the basic notions  and the framework that we use in this paper. That is, we first  describe the
 Doss-Sussmann transformation given in Doss \cite{doss} and Sussmann  \cite{sussmann}, which is the link
between  the stochastic  and  ordinary differential equations (see  Al\`os et al. \cite{alos1}, or Nourdin and Neuenkirch \cite{nourdin},  for  fractional Brownian motion case). Then,  we provide a numerical method and its rate of convergence  for the unique solution of (\ref{eqest}). These are
 the main result of this article (see Theorem \ref{teo1}).

\subsection{Doss-Sussmann transformation}\label{doss-sussmann}

Henceforth, we consider the stochastic differential equation

\begin{equation}\label{eqest2}
X_{t} = x + \int_{0}^{t} b(X_{s})ds + \int_{0}^{t} \sigma(X_{s}) \circ dB_{s}, \quad  t \in [0,T], 
\end{equation}
where $B=\{B_t:t\in[0,T]\}$ is a fractional Brownian motion with Hurst parameter $1/4 < H < 1/2$, $x \in \mathbb{R}$ and the stochastic integral in (\ref{eqest2}) is understood in the sense of Stratonovich, which is introduced  in \cite{alos1}.  Remember that $B$
is defined in (\ref{covariance}). The coefficients  $b,\sigma:\mathbb{R}\rightarrow\mathbb{R}$ are measurable functions  such that
\begin{itemize}
\item [(H)] $b \in C^{2}_{b}(\mathbb{R})$ and  $\sigma \in C^{2}_{b}(\mathbb{R})$.
\end{itemize}
\begin{rem} \label{cotas}
By assumption(H), we have, for $z \in \mathbb{R}$,  
\begin{itemize}
\item $\vert  b(z) \vert \leq M_{1} $,  $\vert  b'(z) \vert \leq M_{4} $ and $\vert  b''(z) \vert \leq M_{6} $.
\item $\vert  \sigma (z) \vert \leq M_{5} $,  $\vert  \sigma ' (z) \vert \leq M_{2} $ and $\vert  \sigma '' (z) \vert \leq M_{3} $.
\end{itemize} 
We explicitly give these constants so that it will be  clear where   we use them  in our analysis.
\end{rem}

Now, we explain the relation between (\ref{eqest2}) and ordinary differential equations: the so call Doss-Sussmann transformation.

In Al\`os et al. (Proposition 6) is proven that the equation (\ref{eqest2}) has a unique solution of the form
\begin{equation}\label{sol}
X_{t} = \phi \left(Y_{t}, B_{t} \right).
\end{equation}
The function  $\phi:\mathbb{R}^2\rightarrow\mathbb{R}$ is  the solution of the ordinary differential equation
\begin{eqnarray}\label{phiori}
{\partial \phi  \over \partial \beta }(\alpha, \beta) &= &\sigma(\phi(\alpha, \beta)),\quad \alpha, \  \beta \in \mathbb{R},\nonumber\\
\phi(\alpha , 0) &=& \alpha  ,
\end{eqnarray}
and  the process $Y$ is the pathwise solution to the equation 	
\begin{equation*}\label{y1}
Y_{t} = x + \int_{0}^{t} \left( {\partial \phi  \over \partial \alpha } (Y_{s}, B_{s})   \right)^{-1} b \left( \phi (Y_{s}, B_{s} ) \right) ds,\quad
t\in[0,T].
\end{equation*}
By Doss \cite{doss}, we have
\begin{equation}\label{eq:phi}
{\partial \phi  \over \partial \alpha }(\alpha, \beta) = exp \left( \int_{0}^{\beta} \sigma'(\phi(\alpha, s )) ds  \right),
\end{equation}
which implies 
\begin{equation}\label{y2}
Y_{t} = x + \int_{0}^{t}  \exp \left( -\int_{0}^{B_{s}} \sigma'(\phi(Y_{s}, u )) du \right)    b \left( \phi (Y_{s}, B_{s} ) \right) ds.
\end{equation}

\subsection{Numerical Method}
In this section, we describe our numerical scheme associated to the unique solution of  (\ref{eqest2}). Towards this end,  in Section \ref{sec:2.2.1},
we first  propose an approximation to the function $\phi$ given in (\ref{eq:phi}), and then, in Section \ref{sec:2.2.2} we approximate the process $Y$.
In both sections we suppose that (H) holds.

\subsubsection{Approximation of $\phi$}\label{sec:2.2.1}
Note that, for $x\in\mathbb{R}$,  equation (\ref{phiori}) has the form
\begin{equation}\label{phi}
\phi(x,u) = x + \int_{0}^{u} \sigma(\phi(x,s))ds.
\end{equation}
For each $l\in\mathbb{N}$, we take the partition $\left\lbrace  u_{i}^{l} , i\in\{-l,\ldots, l\}\right\rbrace$ of the interval $[- \Vert B\Vert_{\infty}, \Vert B\Vert_{\infty}]$ given by $-\Vert B \Vert_{\infty}=u^{l}_{-l} < \ldots <u^{l}_{-1}< u^{l}_{0}=0< u^{l}_{1}< \ldots < u^{l}_{l} = \Vert B \Vert_{\infty}$. Here,  $ \Vert B \Vert_{\infty} =\sup_{t\in[0,T]}|B_t|$,
\begin{equation}
u^{l}_{i+1} = u^{l}_{i} + {\Vert B \Vert_{\infty} \over l} = {(i+1)\Vert B \Vert_{\infty} \over l},\quad u^{l}_{-(i+1)} = u^{l}_{-i} - 
{\Vert B \Vert_{\infty} \over l} =  -{(i+1)\Vert B \Vert_{\infty} \over l}. \nonumber 
\end{equation}

 Let $x\in\mathbb{R}$ be given in (\ref{y2}).  Set
\begin{equation}\label{M}
M:= \vert x \vert + T\left( M_{1}\exp(M_{2} \Vert B \Vert _{\infty}) + \Vert B\Vert_{H-\rho}C_{3}T^{H-\rho} \right) ,
\end{equation}
where $\rho\in(0,H)$,  $\Vert B\Vert_{H-\rho}$ is the $(H-\rho$)-H\"older norm of $B$ on $[0,T]$,
$$C_{3}=  M_{1}M_{2} \exp \left(M_{2}\Vert B \Vert_{\infty} \right) +  M_{4} \exp \left(M_{2} \Vert B \Vert_{\infty} \right) M_{5}
\Vert B \Vert_{\infty}\left( 1 +M_2  \right)
$$
and
 $M_i$, $i\in\{1,\ldots,6\}$ are defined in Remark \ref{cotas}.

Now, we define the function $\phi^{l} : \mathbb{R}^{2} \rightarrow \mathbb{R}$ by
\begin{equation}\label{phi0}
\phi^{l}(z,u) = 0  \ \ \  \mbox{if}  \ \ \ (z,u) \not\in [-M,M] \times [-\Vert B \Vert_{\infty},\Vert B \Vert_{\infty}];
\end{equation}
and, for $k=1, \ldots ,l$,
\begin{equation}\label{phin1}
\phi^{l}(z,u) = \phi^{l}(z,u^{l}_{k-1}) + \int_{u_{k-1}^{l}}^{u} \sigma \left(\phi^{l}(z,u^{l}_{k-1}) + (s-u_{k-1}^{l}) \sigma \left( \phi^{l}(z,u^ {l}_{k-1}) \right) \right) ds,
\end{equation}
if $z\in[-M,M]$ and  $u \in (u^{l}_{k-1} ,u^{l}_{k} ]$, with
\begin{equation}\label{phinicial}
\phi^{l}(z,u^{l}_{0}) = z,\quad\hbox{ if}\ \  z\in[-M,M].
\end{equation}
The definition of $\phi^l$ for the case $k=-l,\ldots, 0$ is similar. That is,
\begin{equation}\label{phin1n}
\phi^{l}(z,u) = \phi^{l}(z,u^{l}_{k}) -\int^{u_{k}^{l}}_{u} \sigma \left(\phi^{l}(z,u^{l}_{k}) + (s-u_{k}^{l}) \sigma \left( \phi^{l}(z,u^ {l}_{k}) \right) \right) ds,
\end{equation}
if $z\in[-M,M]$ and  $u \in [u^{l}_{k-1} ,u^{l}_{k} )$.

Also,  we consider the  function $\Psi^l: \mathbb{R}^{2} \rightarrow \mathbb{R}$, which is equal to  
\begin{equation}\label{psi0}
\Psi^{l}(z,u) = 0  \ \ \  \mbox{if}  \ \ \ (z,u) \not\in [-M,M] \times [-\Vert B \Vert_{\infty},\Vert B \Vert_{\infty}], 
\end{equation}
and, for $k=1, \ldots ,l$,
\begin{eqnarray}
\Psi^{l}(z,u) &=&  \Psi^{l}(z,u^{l}_{k-1}) + \int_{u_{k-1}^{l}}^{u} \left( \sigma \left(  \Psi^{l}(z,u^{l}_{k-1})\right) +  \sigma' \sigma \left( \Psi^{l}(z,u^{l}_{k-1}) \right) (s-u_{k-1}^{l}) \right) ds \nonumber \\  
& = &  \Psi^{l}(z,u^{l}_{k-1}) + (u-u_{k-1}^{l}) \left( \sigma \left(\Psi^{l}(z,u^{l}_{k-1})\right) + \sigma' \sigma \left( \Psi^{l}(z,u^{l}_{k-1}) \right) {(u-u_{k-1}^{l}) \over 2 } \right) \nonumber \\ \label{psin}
\end{eqnarray}
if $z\in[-M,M]$ and  $u \in (u^{l}_{k-1} ,u^{l}_{k} ]$, with
\begin{equation*}\label{psinicial}
\Psi^{l}(z,u^{l}_{0}) = z, \quad\hbox{ if}\ \  z\in[-M,M].
\end{equation*}
For    $k=-l,\ldots, 0$, $\Psi^{l}$ is introduced as
$$
\Psi^{l}(z,u) =  \Psi^{l}(z,u^{l}_{k}) - \int^{u_{k}^{l}}_u \left( \sigma \left(  \Psi^{l}(z,u^{l}_{k})\right) +  \sigma' \sigma \left( \Psi^{l}(z,u^{l}_{k}) \right) (s-u_{k}^{l}) \right) ds ,
$$
If  $z\in[-M,M]$ and  $u \in[u^{l}_{k-1} ,u^{l}_{k} ]$.
From  equation (\ref{psin}) and last equality,  it can be seen that
 $\Psi^l(z, \cdot)$ is continuous on $[-\Vert B \Vert_{\infty},\Vert B \Vert_{\infty}]$.

We remark that the function $\phi^l$ given in (\ref{phin1}) and (\ref{phin1n}) is an auxiliary tool that allows us to use Taylor's theorem in the analysis of the numerical scheme proposed in this paper (i.e., Theorem \ref{teo1}). Indeed, the Taylor's
theorem is utilized in Lemma \ref{lemapsi1}.
\subsubsection{Approximation of $Y$} \label{sec:2.2.2}
Here,  we approximate the  solution of  equation (\ref{y2}).

 For $l\in\mathbb{N}$, we define the process $Y^l$ as the solution of the following ordinary differential equation, where the existence and uniqueness is guaranteed since the coefficient $g^l: \mathbb{R}^2 \to \mathbb{R}$ satisfies
Lipschitz and linear growth conditions in the second variable (see Remark \ref{cotas} and Lemma \ref{lemapsin}).
\begin{eqnarray}\label{yl}
Y^{l}_{t} &=& x + \int_{0}^{t}  g^{l}\left(B_{s} , Y^{l}_{s}\right) ds, \quad  Y^{l}_{0} = x,  
\end{eqnarray}
where
\begin{equation}\label{gl1}
g^{l}\left(B_{s} , Y^{l}_{s}\right) =  \exp \left( -\int_{0}^{B_{s}} \sigma'(\Psi^{l}(Y^{l}_{s}, u )) du \right)    b \left( \Psi^{l}(Y^{l}_{s}, B_{s} ) \right).
\end{equation}
Now, for  $m\in\mathbb{N}$, we set the partition $0=t_{0} < \ldots  <t_{m} = T$ of $[0,T]$ with $t_{i+1} = t_{i} + {T \over m}$ and we define the process $Y^
{l,m}$ as:
\begin{small}
\begin{eqnarray}
Y_{0}^{l,m} &=& x , \nonumber \\ 
Y_{t}^{l,m} &=& Y_{t^{m}_{k}}^{l,m} + \int_{t_{k}^{m}}^{t} \left[ g^{l}\left(	B_{t^{m}_{k}} , Y_{t^{m}_{k}}^{l,m} 	\right) + h_{1}^{l}\left(	B_{t^{m}_{k}} , Y_{t^{m}_{k}}^{l,m} 	\right) \left( B_{s} - B_{t^{m}_{k}}  \right) \right] ds  \label{ynn},
\end{eqnarray}
\end{small}
for $t_{k}^{m} \leq t <  t_{k+1}^{m}$, where 
$h_{1}^{l}(u,z) = {\partial g^{l}(u,z) \over \partial u}$
and $g^{l}$ is given by (\ref{gl1}). So
\begin{equation}\label{dergl}
{\partial g^{l}(u,z) \over \partial u} = - g^{l}(u,z) \sigma'(\Psi^{l}(z,u)) + \exp \left( - \int_{0}^{u} \sigma'( \Psi^{l}(z,r)) dr  \right) b'(\Psi^{l}(z,u)) { \partial \Psi^{l}(z,u) \over \partial u} .
\end{equation}
By Remark \ref{cotas}, we can see
\begin{equation}
\left\vert  g^{l}(u,z) \right\vert \leq M_{1} \exp \left( M_{2} \vert  u\vert  \right). \label{cotagl1}
\end{equation}
Also we have  
\begin{equation}
\vert h_{1}^{l}(B_{t^{m}_{k}},Y^{l,m}_{t^{m}_{k}}) \vert   \leq  C_{3}, \label{cotah1}
\end{equation} 
where   
 $C_{3} $ is given in (\ref{M}). Moreover,
 assuming that  (\ref{cotagl1}) and (\ref{cotah1}) are satisfied, it is not hard to  prove by induction  that
\begin{equation*}
\sup\limits_{t\in[0,T]} \vert Y_{t}^{n,n} \vert \leq \vert  x \vert + T\left(M_{1} \exp(M_{2} \Vert B \Vert_{\infty} )   + T^{H-\rho}  \Vert B \Vert_{H-\rho} C_{3}
\right)=M.
\end{equation*}

Finally, in a similar way to Garz\'on et al.  \cite{garzon},  for $n\in\mathbb{N}$,  we define the approximation $X^n$ of $X$ by:
\begin{equation}\label{metodo}
X_{t}^{n} = \Psi^{n} \left( Y^{n,n}_{t} , B_{t} \right),
\end{equation}
where $\Psi^{n}$ and $Y_{t}^{n,n}$ are given by (\ref{psin}) and (\ref{ynn}), respectively.\\

Now, we are in position to state our main result\\

\begin{teo}
 \label{teo1}
Let  (H) be satisfied and $1/4< H < 1/2 $, then
$$
\left\vert X_{t} - X_{t}^{n} \right\vert  \leq C n^{-2(H-\rho)},  
$$
where $\rho >0 $ is small enough and $C$ is a constant that does not depend on $n$.
\end{teo}
\begin{rem} The constant $C$ has the form
\begin{eqnarray*}
 C &=& \exp(2M_{2} \Vert B \Vert_{\infty}) \left[ C_{2} \exp(C_{1}T) + \frac{ M_{2}^{2} M_{5} \Vert B \Vert^{3}_{\infty} }{6} +  \frac{ M_{5}^{2} M_{3} \Vert B \Vert^{3}_{\infty} }{6} \right. \\
 & + & \left. C_{6} T \exp(C_{7} T) \right],
 \end{eqnarray*}
 with  
{\scriptsize
\begin{eqnarray}
C_{1} &=&  (M_{4} +  M_{1}M_{3}  \Vert B \Vert_{\infty} )  \exp( M_2 (\Vert B \Vert_{\infty} + T)), \nonumber \\
C_{2} &=&  \exp( M_2 \Vert B \Vert_{\infty})  (M_{4} +  M_{1}M_{3}\Vert B \Vert_{\infty} )({M_2^2 M_5 \Vert B \Vert^{3}_{\infty} \over 6} \exp(M_2 \Vert B \Vert_{\infty})  + {M_3 M_5^2  \Vert B \Vert^{3}_{\infty} \over 6} \exp(2M_2 \Vert B \Vert_{\infty}) ), \nonumber \\
C_{3} &=&  M_{1}M_{2} \exp \left(M_{2}\Vert B \Vert_{\infty} \right) +  M_{4} \exp \left(M_{2} \Vert B \Vert_{\infty} \right) M_{5}
\Vert B \Vert_{\infty}\left( 1 +M_2  \right),  \nonumber \\
C_4 &=&  M_{1} \exp \left( M_{2} \Vert B \Vert_{\infty} \right) + C_3 T^{H - \rho} \Vert B \Vert_{H-\rho}, \nonumber \\
C_{5} &=&  \exp (M_{2} \Vert B \Vert_{\infty}) \left[   \Vert B \Vert_{\infty} (1+M_{2}) \left( M_{3} M_{1}M_{5}  + M_{2} M_{4} M_{5}  + M_{6}M_{5}  \Vert B \Vert_{\infty} (1+ M_{2})  \right)  \right. \nonumber  \\ 
&+& \left. M_{1}M_{2} + M_{4}M_{5}( 1+ M_{2}) \right], \nonumber \\
C_{6} &=& \left[ C_{4} \exp(3 M_{2} \Vert B \Vert_{\infty} ) [M_{4} + M_{1}M_{3}\Vert B \Vert_{\infty} ]T^{1-2(H - \rho)}  + (C_{5} + C_{8}) \Vert B \Vert_{H-\rho} \right], \nonumber \\
C_{7} &=& \exp(3 M_{2} \Vert B \Vert_{\infty} ) [M_{4} + M_{1} M_{3} \Vert B \Vert_{\infty}],  \nonumber  \\
C_{8} & = &  M_{4} M_{2}  \exp(M_{2} \Vert B \Vert_{\infty}) \left[ (M_{5} + M_{5}M_{2})  \Vert B \Vert_{\infty} +  M_{5}M_{2} \right]. \nonumber
\end{eqnarray}
}
\end{rem}

\begin{rem}\label{MR}
We choose the constant $M$ because the processes given in (\ref{yl}) and (\ref{ynn}), as well as the solution to
(\ref{y2}), are bounded by $M$, as it is pointed out in this section.
\end{rem}

\section{Preliminary lemmas}\label{sec3}

In this section, we stated the auxiliary tools that  we need in order to prove our main
result Theorem \ref{teo1}. The first four lemmas are related to the apriori estimates of $\phi$. We recall you that the constants $M_{i}, i \in \{1, \ldots, 6 \}$ are introduced in Remark \ref{cotas}.

\begin{lem}\label{lemaphi1}
Let $\phi$ and $\phi^{l}$ be given by (\ref{phi}) and (\ref{phin1}), respectively. Then, Hypothesis (H) implies that , for $ (z,u) \in [-M,M] \times [-\Vert B \Vert_{\infty},\Vert B \Vert_{\infty}]$, we have
\begin{equation*}\label{3.1}
\left\vert  \phi(z,u) - \phi^{l}(z,u)  \right\vert   \leq  {M_{2}^{2} M_{5} \Vert  B\Vert_{\infty}^{3} \over {6 l^2}} \exp \left( M_{2} \Vert  B\Vert_{\infty} \right), 
\end{equation*}
\end{lem}

\begin{lem}\label{lemapsi1}
Let $\phi^{l}$ and $\Psi^{l}$ be given by (\ref{phin1}) and (\ref{psin}), respectively. Then, Hypothesis (H) implies
\begin{equation*}
\left\vert \phi^{l}(z,u) - \Psi^{l}(z,u) \right\vert \leq {M_{3} M_{5}^{2} \Vert B \Vert_{\infty}^{3} \over 6 l^2} \exp \left(2 M_{2} \Vert B \Vert_{\infty} \right),  
\end{equation*}
for $ (z,u) \in [-M,M] \times [-\Vert B \Vert_{\infty},\Vert B \Vert_{\infty}]$. 
\end{lem}
\begin{lem}\label{lemapsin}
Let  $\Psi^{l}$ be introduced in (\ref{psin}) and Hypothesis (H) hold. Then, for $(z_1, z_2,u) \in [-M,M]^2 \times [-\Vert B \Vert_{\infty},\Vert B \Vert_{\infty}]  $, 
\begin{equation*}
\left\vert \Psi^{l}(z_{1},u) -   \Psi^{l}(z_{2},u)   \right\vert  \leq  \left\vert z_{1} -z_{2} \right\vert \exp(2M_2 \Vert B \Vert_{\infty} ). 
\end{equation*}
\end{lem}

\begin{lem}\label{difphin1}
Let $\phi^{l}$ be given in (\ref{phin1}). Then, under Hypothesis (H),   
\begin{equation}\label{difphin1-1}
\vert \phi^{l}(z_{1},u) - \phi^{l}(z_{2},u) \vert \leq \vert z_{1}-z_{2} \vert \exp(2M_2 \Vert B \Vert_{\infty}  ),
\end{equation}
for $(z_1, z_2,u) \in [-M,M]^2 \times [-\Vert B \Vert_{\infty},\Vert B \Vert_{\infty}] $. 
\end{lem}


  Now we proceed to state the lemmas referred to the estimates on $Y^n - Y$.
\begin{lem}\label{lemayyn}
Assume that Hypothesis (H) is satisfied. Let $Y$ and $Y^{n}$ be given in (\ref{y2}) and (\ref{yl}), respectively. Then, for $t \in [0,T]$, 
\begin{equation*}
\left\vert Y_{t} - Y_{t}^{n} \right\vert  \leq \exp (C_1 T) {C_2 \over n^{2}},
\end{equation*}
where 
\begin{equation*}
C_{1} =  (M_{4} +  M_{1}M_{3}  \Vert B \Vert_{\infty} )  \exp( M_2 (\Vert B \Vert_{\infty} + T))
\end{equation*}
and 
\begin{eqnarray*}
C_{2} &=&  \exp( M_2 \vert B \Vert_{\infty})  (M_{4} +  M_{1}M_{3}\Vert B \Vert_{\infty}) \\
& \times & \left({T M_2^2 M_5 \Vert B \Vert_{\infty}^{3} \over 6} \exp(M_2 \Vert B \Vert_{\infty})  + {T M_3 M_5^2  \Vert B \Vert_{\infty}^3 \over 6} \exp(2M_2 \Vert B \Vert_{\infty}) \right).
\end{eqnarray*}

\end{lem}

\begin{lem}\label{difynn}
Let $Y^{n,n}$ be defined in (\ref{ynn}). Then Hypothesis (H) implies, for $s \in (t_{k}^{n}, t_{k+1}^{n}]$, 
\begin{equation*}
\vert Y^{n,n}_{s} - Y^{n,n}_{t^{n}_{k}} \vert  \leq C_{4} (s-t_{k}^{n}), 
\end{equation*}
where 
$C_4 =  M_{1} \exp \left( M_{2} \Vert B \Vert_{\infty} \right) + C_3 T^{H - \rho} \Vert B \Vert_{H-\rho}$ and
\begin{equation}
C_{3} =  M_{1}M_{2} \exp \left(M_{2}\Vert B \Vert_{\infty} \right) +  M_{4} \exp \left(M_{2} \Vert B \Vert_{\infty} \right) M_{5}
\Vert B \Vert_{\infty}\left( 1 +M_2  \right).  \nonumber 
\end{equation}
\end{lem}

\begin{lem}\label{difynynn}
 Suppose that Hypothesis (H) holds. Let $Y^{n}$ and $Y^{n,n}$ be given in (\ref{yl}) and (\ref{ynn}), respectively. Then,  

 \begin{equation*}
\left\vert Y^{n}_{t} - Y_{t}^{n,n} \right\vert  \leq   C_{6} T \left({T \over n} \right)^{2(H-\rho)}  \exp(C_{7}T), \quad t \in [0,T],
\end{equation*}
where $0<\rho <H$, 
\begin{eqnarray*}
C_{5} &=&  \exp (M_{2} \Vert B \Vert_{\infty}) \left[   \Vert B \Vert_{\infty} (1+M_{2}) \left( M_{3} M_{1}M_{5}  + M_{2} M_{4} M_{5}  + M_{6}M_{5}  \Vert B \Vert_{\infty} (1+ M_{2})  \right)  \right. \\ 
&+& \left. M_{1}M_{2} + M_{4}M_{5}( 1+ M_{2}) \right],
\end{eqnarray*}
\begin{equation}
C_{6} = \left[ C_{4} \exp(3 M_{2} \Vert B \Vert_{\infty} ) [M_{4} + M_{1}M_{3}\Vert B \Vert_{\infty} ]T^{1-2(H - \rho)}  + (C_{5} + C_{8}) \Vert B \Vert_{H-\rho} \right] , \label{c6}  
\end{equation}
with $C_{4}$ given in Lemma \ref{difynn}, and
\begin{equation}\label{c7}
C_{7} =  \exp(3 M_{2} \Vert B \Vert_{\infty} ) [M_{4} + M_{1} M_{3} \Vert B \Vert_{\infty}],
\end{equation}
\begin{equation*}
C_{8} =  M_{4} M_{2}  \exp(M_{2} \Vert B \Vert_{\infty}) \left[ (M_{5} + M_{5}M_{2})  \Vert B \Vert_{\infty} +  M_{5}M_{2} \right]. 
\end{equation*}
\end{lem}


\section{Convergence of the Scheme: Proof of Theorem \ref{teo1}}\label{sec4}
We are now ready to prove the main result of this article, which gives a theoretical bound on the speed of convergence for $X^n$ defined in (\ref{metodo}). Remember that the constants $M_{i}, i \in \{1, \ldots, 6 \}$, are given in Remark \ref{cotas}.

\begin{proof}

By (\ref{sol}) and (\ref{metodo}), we have, for  $t\in [0,T]$, 
\begin{equation*}
\vert X_{t} - X_{t}^{n} \vert  \leq H_{1}(t) + H_{2}(t) + H_{3}(t),
\end{equation*}
where
\begin{small}
\begin{eqnarray}
H_{1}(t) &=&  \vert \phi\left( Y_{t}, B_{t}   \right) - \phi^{n} ( Y^{n}_{t}, B_{t} ) \vert  \nonumber \\ 
H_{2}(t) &=&  \vert \phi^{n} ( Y^{n}_{t}, B_{t} )  - \phi^{n} ( Y^{n,n}_{t}, B_{t} )       \vert  \nonumber \\ \nonumber
H_{3}(t) &=&  \vert \phi^{n} ( Y^{n,n}_{t}, B_{t} )  - \Psi^{n} ( Y^{n,n}_{t}, B_{t} )\vert. \nonumber
\end{eqnarray}
\end{small}
Now we proceed to obtain estimates of $H_{1}$, $H_{2}$ and $H_{3}$. By property (\ref{eq:phi}), we get
 
\begin{small}
\begin{eqnarray*}
H_{1}(t) & \leq & \vert \phi\left( Y_{t}, B_{t}   \right) - \phi ( Y^{n}_{t}, B_{t} ) \vert  
+ \vert \phi ( Y^{n}_{t}, B_{t} ) - \phi^{n} ( Y^{n}_{t}, B_{t} ) \vert  \nonumber \\ 
& \leq & \exp \left(M_{2} \Vert B \Vert_{\infty} \right) \left\vert Y_{t} - Y_{t}^{n} \right\vert  
+ \vert \phi ( Y^{n}_{t}, B_{t} ) - \phi^{n} ( Y^{n}_{t}, B_{t} ) \vert. \nonumber \\ 
\end{eqnarray*}
\end{small}
Therefore, by Lemmas \ref{lemaphi1} and \ref{lemayyn}  
\begin{equation}
H_{1}(t) \leq \exp \left(M_{2} \Vert B \Vert_{\infty} \right) \exp(C_{1}T) {C_{2} \over n^2 }+  {M_{2}^{2} M_{5} \Vert B \Vert_{\infty}^{3} \over 6n^2 }    \exp \left( M_{2} \Vert B \Vert_{\infty} \right). \label{h1}
\end{equation}

Also Lemmas \ref{difphin1} and \ref{difynynn}, yield
\begin{small}
\begin{eqnarray}
H_{2}(t) & \leq & \exp(2M_{2} \Vert B \Vert_{\infty})  \vert  Y^{n}_{t} - Y^{n,n}_{t} \vert  \nonumber \\
& \leq &\exp(2M_{2} \Vert B \Vert_{\infty})C_{6} T \left({T \over n} \right)^{2(H-\rho)}  \exp(C_{7}T).\label{h2}
\end{eqnarray}
\end{small}
For $H_{3}$, we use Lemma \ref{lemapsi1}. So
\begin{equation}
H_{3}(t) \leq {M_{3} M_{5}^{2} \Vert B \Vert_{\infty}^3 \over 6n^2}  \exp(2M_{2}\Vert B \Vert_{\infty}).\label{h3}
\end{equation}

Finally, from (\ref{h1}) to (\ref{h3}), we have
\begin{equation*}
\vert X_{t} - X_{t}^{n} \vert  \leq C n^{-2(H -\rho)  }  ,
\end{equation*}  
which shows that Theorem \ref{teo1} holds.
\end{proof}

\section{Proofs of preliminary lemmas}\label{sec5}
Here, we provide the proofs of Lemmas \ref{lemaphi1} to \ref{difynynn}. First, we will prove, by induction, that the  statements of  Lemmas \ref{lemaphi1} to \ref{difphin1} hold   for all $k=1,2, \ldots l$ and  $u \in (u_{k-1}^{l} ,  u^{l}_{k}]$. We will consider for simplicity the case $u>0$, the other case can be treated similarly.  
\subsection*{Proof of Lemma \ref{lemaphi1}}

\begin{proof} Let $z\in[-M,M]$.
We will prove by induction that, for all $k \in \{1,\ldots, l\}$ and  $u \in (u_{k-1}^{l} ,  u^{l}_{k}]$, we have 
\begin{equation}\label{philema1}
\left\vert  \phi(z,u) - \phi^{l}(z,u)  \right\vert   \leq  {M_{2}^{2} M_{5} \Vert  B\Vert_{\infty}^{3} \over 6 l^3} \tilde{C}_k,
\end{equation}
where $\tilde{C}_k = \exp \left( M_{2} (u^{l}_{k} - u^{l}_{0})\right) + \ldots + \exp \left( M_{2} (u^{l}_{k} - u^{l}_{k-1})\right).$
As a consequence
 we obtain the  global bound
\begin{equation}\label{philema}
\left\vert \phi(z,u) - \phi^{l}(z,u) \right\vert \leq  {M_{2}^{2} M_{5} \Vert  B\Vert_{\infty}^{3} \over 6 l^2}  \exp \left( M_{2} \Vert  B\Vert_{\infty} \right),
\end{equation}
where $M_2$ and $M_5$ are constants independent of $k$ and they are given in Remark \ref{cotas}.

First for $k=1$, let  $0=u_{0}^{l} < u \leq u^{l}_{1}$, then (\ref{phi}), (\ref{phin1}),
 the Lipschitz condition on $\sigma$ (with constant $M_2$) and the fact that $\phi(z,u_0^l) = \phi^l(z,u_0^l)=z$ imply
\begin{small}
\begin{eqnarray}
\left\vert  \phi(z,u) - \phi^{l}(z,u) \right\vert  & \leq & M_{2} \int_{u_{0}^{l}}^{u} \left\vert \phi(z,s) - z - (s-u_{0}^{l}) \sigma \left( z \right) \right\vert ds    \nonumber \\  
& \leq &   M_{2} \int_{u_{0}^{l}}^{u}\left\vert \phi(z,s) - \phi^{l}(z,s) \right\vert ds + M_{2}\int_{u_{0}^{l}}^{u} \left\vert \phi^{l}(z,s) -z-(s-u_{0}^{l})\sigma(z)\right\vert ds \nonumber \\ 
& = & M_{2}\int_{u_{0}^{l}}^{u} \left\vert \phi(z,s) - \phi^{l}(z,s) \right\vert ds +    M_{2} \bold{I}_{0}^{l}. \label{i0}  
\end{eqnarray}
\end{small} 
Next, we bound the term $\bold{I}_{0}^{l} $, 
\begin{small}
\begin{equation*}
\bold{I}_{0}^{l}= \int_{u_{0}^{l}}^{u} \left\vert \phi^{l}(z,s) -z-(s-u_{0}^{l})\sigma(z)\right\vert ds  =  \int_{u_{0}^{l}}^{u} \left\vert \phi^{l}(z,s) -z- \int_{u_0^l}^s \sigma(z)dr \right\vert ds.
\end{equation*}
From (\ref{phin1}), the Lipschitz condition and the bound on $\sigma$, we get
\begin{eqnarray}
\bold{I}_{0}^{l}& = &  \int_{u_{0}^{l}}^{u} \left\vert \int_{u_{0}^{l}}^{s}  \sigma \left(z+(r-u_{0}^{l})\sigma(z) \right)dr  - \int_{u_{0}^{l}}^{s} \sigma(z) dr \right\vert ds \nonumber \\
& \leq &  \int_{u_{0}^{l}}^{u}\int_{u_{0}^{l}}^{s} \left\vert \sigma \left(z + (r-u_{0}^{l})\sigma(z)\right) -\sigma(z)\right\vert dr ds \nonumber \\ 
& \leq & M_{2}\int_{u_{0}^{l}}^{u}\int_{u_{0}^{l}}^{s}(r-u_{0}^{l})\left\vert \sigma(z) \right\vert dr ds \leq  { M_{2} M_{5}
 \Vert B\Vert_{\infty}^3 \over 6 l^3}.  \label{I0} 
\end{eqnarray}
\end{small}
Therefore by (\ref{i0}), (\ref{I0}) and the Gronwall lemma we obtain
\begin{small}
\begin{equation*}
\left\vert  \phi(z,u) - \phi^{l}(z,u)  \right\vert   \leq { M_{2}^{2} M_{5} \Vert B\Vert_{\infty}^3  \over 6 l^3} \exp \left( M_{2} (u_1^l - u_0^l)\right), \quad \mbox{for} \ u \in (0 , u_{1}^{l}]. \label{phid0} 
\end{equation*}
\end{small} 

Now, consider an index $k \in \{2,\ldots, l\}$. Our induction assumption
is that (\ref{philema1}) is true for $u \in (u_{k-1}^{l},u^{l}_{k}]$. We shall now propagate the induction, that is prove that the inequality is also true for its successor, 
$k+1$. We will thus study (\ref{philema1}) for $u \in (u_{k}^{l} ,u^{l}_{k+1}]$. Following (\ref{phi}),  (\ref{phin1}) and our induction hypothesis we  establish 
\begin{small}
\begin{eqnarray}
\left\vert  \phi(z,u) - \phi^{l}(z,u)  \right\vert & \leq & \left\vert \phi(z,u^{l}_{k}) - \phi^{l}(z,u^{l}_{k})  \right\vert  \nonumber \\ 
& + & \int_{u^{l}_{k}}^{u} \left\vert \sigma \left( \phi(z,s) \right) -\sigma \left(\phi^{l}(z,u^{l}_{k}) + (s- u^{l}_{k}) \sigma \left( \phi^{l}(z,u^{l}_{k}) \right) \right) \right\vert ds, \nonumber \\
 & \leq &  \tilde{C}_k {M_2^2 M_5 \Vert B\Vert_{\infty}^3 \over 6{l}^{3}} +  \int_{u^{l}_{k}}^{u} \left\vert   \sigma \left( \phi(z,s) \right)-\sigma \left(\phi^{l}(z,u^{l}_{k}) + (s- u^{l}_{k}) \sigma \left( \phi^{l}(z,u^{l}_{k}) \right) \right) \right\vert ds. \nonumber
 \end{eqnarray}
From Lipschitz condition on $\sigma$, 
 \begin{eqnarray}
\left\vert \phi(z,u) - \phi^{l}(z,u) \right\vert & \leq & \tilde{C}_k {M_2^2 M_5 \Vert B \Vert^3_{\infty} \over 6{l}^{3}}+ M_{2} \int_{u^{l}_{k}}^{u}\left\vert \phi(z,s)-\phi^{l}(z,s) \right\vert ds \nonumber \\ 
& + & M_{2}\int_{u^{l}_{k}}^{u}\left\vert \phi^{l}(z,s)-\phi^{l}(z,u^{l}_{k})- (s- u^{l}_{k}) \sigma\left(\phi^{l}(z,u^{l}_{k})\right)\right\vert ds \nonumber \\ 
& = & \tilde{C}_k {M_2^2 M_5 \Vert B \Vert^3_{\infty}  \over 6{l}^{3}} + M_{2} \int_{u^{l}_{k}}^{u} \left\vert    \phi(z,s)-\phi^{l}(z,s) \right\vert ds  + M_{2} {\bold I}^{l}_{k},  \label{ik}
\end{eqnarray}
\end{small}
where $\tilde{C}_k = \exp \left( M_{2} (u^{l}_{k} - u^{l}_{0})\right) +\ldots + \exp \left( M_{2} (u^{l}_{k} - u^{l}_{k-1})\right).$

Now, we analyze the term $ {\bold I}^{l}_{k}$, given in equation (\ref{ik}). 
From (\ref{phin1}), the Lipschitz condition and the bound on $\sigma$ we obtain  
\begin{small}
\begin{eqnarray}
{\bold I}^{l}_{k} & \leq &  \int_{u^{l}_{k}}^{u} \int_{u^{l}_{k}}^{s} \left\vert \sigma \left(\phi^{l}(z,u^{l}_{k}) + (r- u^{l}_{k}) \sigma \left( \phi^{l}(z,u^{l}_{k}) \right) \right) - \sigma \left(\phi^{l}(z,u^{l}_{k}) \right) \right\vert dr ds \nonumber \\ 
& \leq & M_{2} \int_{u^{l}_{k}}^{u} \int_{u^{l}_{k}}^{s} (r - u^{l}_{k}) \left\vert \sigma \left(  \phi^{l}(z,u^{l}_{k}) \right) \right\vert drds  \leq {M_2 M_5 \Vert B \Vert^3_{\infty} \over 6{l}^{3}}. \label{IK}
\end{eqnarray}
\end{small}
Therefore inequalities (\ref{ik}) and (\ref{IK}) yield
\begin{small}
\begin{eqnarray}
\left\vert \phi(z,u) - \phi^{l}(z,u)  \right\vert & \leq & \tilde{C}_k {M_2^2 M_5 \Vert B \Vert^3_{\infty} \over 6{l}^{3}} +{M_2^2 M_5 \Vert B \Vert^3_{\infty} \over 6{l}^{3}} +  M_{2} \int_{u^{l}_{k}}^{u} \left\vert\phi(z,s)-\phi^{l}(z,s) \right\vert ds.
\nonumber
\end{eqnarray}
\end{small}

Thus, the Gronwall lemma allows us to establish

\begin{eqnarray}
 & & \left\vert \phi(z,u) - \phi^{l}(z,u)  \right\vert \leq  \left(\tilde{C}_k + 1 \right)\left( {M_2^2 M_5 \Vert B \Vert^3_{\infty}  \over 6{l}^{3}}   \right) \exp(M_2(u^{l}_{k+1} - u^l_k))    
\nonumber \\
&=& \left( \exp \left( M_{2} (u^{l}_{k} - u^{l}_{0})\right) +\ldots + \exp \left( M_{2} (u^{l}_{k} - u^{l}_{k-1}) \right) +1 \right)  \exp(M_2(u^{l}_{k+1} - u^{l}_k)    \left( {M_2^2 M_5 \Vert B \Vert^3_{\infty} \over 6{l}^{3}}   \right)
\nonumber \\ \nonumber \\
&=& \tilde{C}_{k+1} {M_2^2 M_5 \Vert B \Vert^3_{\infty} \over 6{l}^{3}}, \nonumber
\end{eqnarray}
which shows that (\ref{philema1}) is satisfied for $k+1$. 

Now, we prove  that (\ref{philema}) is true.  For all $(z,u) \in [-M,M] \times [-\Vert B \Vert_{\infty},\Vert B \Vert_{\infty}]$, there is some $k \in \lbrace 1,\ldots, l \rbrace$ such that 
$u^{l}_{k-1} < u \leq u^{l}_{k} $ and by the previous calculations 
\begin{small}
\begin{eqnarray}
\left\vert  \phi(z,u) - \phi^{l}(z,u)  \right\vert   & \leq & { M_{2}^{2} M_{5} \Vert B \Vert^3_{\infty} \over 6 l^3} \left[ \exp \left( M_{2} (u^{l}_{k} - u^{l}_{0})\right) +\ldots + \exp \left( M_{2} (u^{l}_{k} - u^{l}_{k-1})\right) \right]
\nonumber \\ 
& \leq & { M_{2}^{2} M_{5} \Vert B \Vert^3_{\infty} \over 6 l^3} k \exp \left( {M_{2} \Vert B \Vert_{\infty}} \right)\nonumber \\ 
& \leq & { M_{2}^{2} M_{5} \Vert B \Vert^3_{\infty} \over 6 l^2}  \exp \left( {M_{2} \Vert B \Vert_{\infty}}\right), \nonumber
\end{eqnarray}
\end{small}
proving the lemma.
\end{proof}

\subsection*{Proof of Lemma \ref{lemapsi1}}
\begin{proof}
As in the proof of Lemma \ref{lemaphi1}, we will prove by induction that, for all $k \in \{1,\ldots, l\}$ and  $u \in (u_{k-1}^{l} ,  u^{l}_{k}]$, we have 
\begin{equation}\label{lemadifl-1}
\left\vert \phi^{l}(z,u) -\Psi^{l}(z,u)\right\vert\leq {M_{3} M_{5}^{2} \Vert B \Vert_{\infty}^{3} \over 6 l^3} k \left( 1 + A_l \right) ^k,  
\end{equation}
with $A_l =  1 + {M_2 \Vert B \Vert_{\infty} \over l} + {1 \over 2}  ({M_2 \Vert B \Vert_{\infty} \over l})^2$. Hence,  
\begin{equation}\label{lemadifl}
\left\vert \phi^{l}(z,u) -\Psi^{l}(z,u)\right\vert\leq {M_{3} M_{5}^{2} \Vert B \Vert_{\infty}^{3} \over 6 l^2} \exp(2 M_2 \Vert B \Vert_{\infty}).  
\end{equation}
where $M_2$, $M_3$ and $M_5$ are constants independent on $k$ and are given in Remark \ref{cotas}.

We first assume that $k=1$.  If $0=u_{0}^{l} < u \leq u^{l}_{1}$ and from equalities (\ref{phin1}) to (\ref{psin}) we obtain that
\begin{small}
\begin{eqnarray*}
\left\vert \phi^{l}(z,u) - \Psi^{l}(z,u)\right\vert & \leq & \int_{u_{0}^{l}}^{u} \left\vert \sigma \left[ 
 \phi^{l}\left(z,u^{l}_{0} \right) + (s-u^{l}_{0}) \sigma \left( \phi^{l}\left(z,u^{l}_{0} \right) \right) \right] \right. \nonumber \\ 
& - &  \left. \left[ \sigma \left( \Psi^{l}\left(z,u^{l}_{0} \right) \right) +  \sigma \sigma' \left( \Psi^{l}\left(z,u^{l}_{0} \right) \right) (s-u^{l}_{0})  \right] \right\vert  ds \nonumber \\
& = & \int_{u_{0}^{l}}^{u} \left\vert \sigma \left( z + (s-u^{l}_{0}) \sigma \left( z \right) \right) - \sigma(z) - \sigma'(z) (s-u^{l}_{0})\sigma(z) \right\vert ds. 
\end{eqnarray*}
\end{small}
By Taylor's theorem there exists a point $\theta \in \left( \inf \{z , z + (s-u^{l}_{0}) \sigma (z) \} , \sup  \{z , z + (s-u^{l}_{0}) \sigma (z) \} \right)$ such that 
\begin{small}
\begin{eqnarray*}
\left\vert \phi^{l}(z,u) -   \Psi^{l}(z,u)   \right\vert   & \leq & \int_{u_{0}^{l}}^{u}  { \left\vert \sigma '' ( \theta )  \right\vert \over 2 } (s-u^{l}_{0})^{2} \left\vert \sigma (z)  \right\vert^{2} ds \nonumber \\ 
& \leq & {M_{3}M_{5}^{2} \Vert B \Vert^3_{\infty} \over 6 l^3}
 \nonumber \\ 
& \leq & {M_{3}M_{5}^{2} \Vert B \Vert^3_{\infty} \over 6 l^3} 1 ( 1 + A_l) ^1.
\end{eqnarray*}
\end{small}
Now,let us consider $k \in \{2,\ldots, l\}$. Our induction assumption is that (\ref{lemadifl-1}) is true for $u \in (u_{k-1}^{l},u^{l}_{k}]$. We will thus study (\ref{lemadifl-1}) for  $u \in (u_{k}^{l} ,  u^{l}_{k+1}]$. Following equations (\ref{phin1}) to (\ref{psin}) and our induction hypothesis, we get
\begin{small}
\begin{eqnarray}
& &\left\vert \phi^{l}(z,u) - \Psi^{l}(z,u)\right\vert \nonumber \\
&\leq & \left\vert \phi^{l}(z,u^{l}_{k}) - \Psi^{l}(z,u^{l}_{k})\right\vert + \int_{u_{k}^{l}}^{u}\left\vert \sigma \left[\phi^{l}\left(z,u^{l}_{k}\right) + (s-u^{l}_{k}) \sigma \left( \phi^{l}\left(z,u^{l}_{k} \right) \right) \right] \right. \nonumber \\ 
 & - & \left. \left[ \sigma \left( \Psi^{l}\left(z,u^{l}_{k} \right) \right) +  \sigma' \left( \Psi^{l}\left(z,u^{l}_{k} \right) \right) (s-u^{l}_{k})\sigma \left( \Psi^{l}\left(z,u^{l}_{k} \right) \right)  \right] \right\vert  ds \nonumber \\ 
& \leq & {M_{3} M_{5}^{2} \Vert B \Vert^{3}_{\infty}  \over 6 l^3}   k \left( 1 + A_l \right) ^k  \nonumber \\ 
& + & \int_{u_{k}^{l}}^{u} \left\vert \sigma \left[ 
 \phi^{l}\left(z,u^{l}_{k} \right) + (s-u^{l}_{k}) \sigma \left( \phi^{l}\left(z,u^{l}_{k} \right) \right) \right]  - \sigma \left[ 
 \Psi^{l}\left(z,u^{l}_{k} \right)  + (s-u^{l}_{k}) \sigma \left( \Psi^{l}\left(z,u^{l}_{k} \right) \right) \right] \right\vert ds \nonumber \\ 
& + &  \int_{u_{k}^{l}}^{u} \left\vert \sigma \left[ 
 \Psi^{l}\left(z,u^{l}_{k} \right)  + (s-u^{l}_{k}) \sigma \left( \Psi^{l}\left(z,u^{l}_{k} \right) \right) \right] -   \left[ \sigma \left( \Psi^{l}\left(z,u^{l}_{k} \right) \right) +  \sigma \sigma' \left( \Psi^{l}\left(z,u^{l}_{k} \right) \right) (s-u^{l}_{k})  \right] \right\vert  ds \nonumber \\
&= & {M_{3} M_{5}^{2} \Vert B \Vert_{\infty}^3 \over 6 l^3}   k \left( 1 + A_l \right) ^k  + \bold{J}_1^k +\bold{J}_2^k  \nonumber
\end{eqnarray} 
 
 From the Lipschitz condition on $\sigma$, and our induction hypothesis
\begin{eqnarray}
\bold{J}_1^k & \leq & M_2 \int_{u_{k}^{l}}^{u} \left\vert \phi^{l}\left(z,u^{l}_{k} \right) -  \Psi^{l}\left(z,u^{l}_{k} \right) \right\vert ds  + M_2^2 \int_{u_{k}^{l}}^{u} (s-u_k^l) \left\vert \phi^{l}\left(z,u^{l}_{k} \right) -  \Psi^{l}\left(z,u^{l}_{k} \right) \right\vert ds \nonumber \\
& \leq & {M_2 M_{3} M_{5}^{2} \Vert B \Vert_{\infty}^3  \over 6 l^3}  k \left( 1 + A_l \right) ^k (u - u^{l}_{k}) +  {M_2^2 M_{3} M_{5}^{2} \Vert B \Vert_{\infty}^3  \over 12 l^3}  k \left( 1 + A_l\right) ^k (u - u^{l}_{k})^2.  \nonumber
\end{eqnarray} 
By Taylor's theorem there exists a point 
\begin{equation*}
\theta_k \in \left( \inf \{ \Psi^{l}(z,u^{l}_{k}), \Psi^{l}(z,u^{l}_{k}) + (s-u^{l}_{k}) \sigma (\Psi^{l}(z,u^{l}_{k})) \} , \sup \{ \Psi^{l}(z,u^{l}_{k}), \Psi^{l}(z,u^{l}_{k}) + (s-u^{l}_{k}) \sigma (\Psi^{l}(z,u^{l}_{k})) \}  \right) 
\end{equation*}
such that 
\begin{eqnarray}
\bold{J}_2^k & \leq &  \int_{u_{k}^{l}}^{u} {\vert \sigma''(\theta_k) \vert  \over 2}  \left\vert \sigma \left( \Psi^{l} \left(z,u^{l}_{k} \right) \right) \right\vert ^2 (s-u^{l}_{k})^{2} ds \nonumber \leq  { M_3 M_5^2 \over 6} (u - u^{l}_{k} )^3 .  \nonumber 
\end{eqnarray}
\end{small}
Therefore 
\begin{eqnarray}
& & \left\vert \phi^{l}(z,u) - \Psi^{l}(z,u)\right\vert \leq 
{M_{3} M_{5}^{2} \Vert B \Vert_{\infty}^3 \over 6 l^3}  k \left( 1 +A_l \right) ^k  +
{M_2 M_{3} M_{5}^{2} \Vert B \Vert_{\infty}^3 \over 6 l^3}  k \left( 1 + A_l \right) ^k (u - u^{l}_{k}) \nonumber \\
&+&   {M_2^2 M_{3} M_{5}^{2} \Vert B \Vert_{\infty}^3 \over 12 l^3}  k \left( 1 + A_l \right) ^k (u - u^{l}_{k})^2 +
{ M_3 M_5^2 \over 6} (u - u^{l}_{k} )^3. \nonumber 
\end{eqnarray}
Since $(u - u^{l}_{k}) \leq {\Vert B \Vert_{\infty} \over l}$ we obtain 

\begin{eqnarray}
& & \left\vert \phi^{l}(z,u) - \Psi^{l}(z,u)\right\vert \leq 
{M_{3} M_{5}^{2} \Vert B \Vert_{\infty}^3 \over 6 l^3} \left[ k\left(1 + A_l \right)^k  + {M_2 \Vert B \Vert_{\infty} \over l}  k\left(1 + A_l \right)^k+   {M_2^2 \Vert B \Vert_{\infty}^2 \over 2l^2 }k\left(1 + A_l \right)^k+ 1\right]. 
\nonumber \end{eqnarray}

Since  $ 1 < \left(1 + A_l \right)^{k+1}$, then
\begin{eqnarray}
& & \left\vert \phi^{l}(z,u) - \Psi^{l}(z,u)\right\vert \leq 
{M_{3} M_{5}^{2} \Vert B \Vert_{\infty}^3 \over 6 l^3} \left[ k\left(1 + A_l \right)^{k+1}  + 1\right] \nonumber \\
&\leq &
{M_{3} M_{5}^{2} \Vert B \Vert_{\infty}^3 \over 6 l^3} \left[ k\left(1 + A_l \right)^{k+1}  +  \left(1 + A_l \right)^{k+1}\right] = {M_{3} M_{5}^{2} \Vert B \Vert_{\infty}^3 \over 6 l^3} (k+1) \left(1 + A_l \right)^{k+1}.
\nonumber \end{eqnarray}
Thus (\ref{lemadifl-1}) holds for any $k \in \{1, \ldots ,l\}$.

Finally, we see that (\ref{lemadifl}) is satisfied. For all $(z,u) \in [-M,M] \times [-\Vert B \Vert_{\infty},\Vert B \Vert_{\infty}]$, there is some $k \in \lbrace 1,\ldots, l \rbrace$ such that 
$u^{l}_{k-1} < u \leq u^{l}_{k} $ and by (\ref{lemadifl-1}),
\begin{small}
\begin{eqnarray}
\left\vert  \phi(z,u) - \Psi^{l}(z,u)  \right\vert   & \leq &  {M_{3} M_{5}^{2} \Vert B \Vert_{\infty}^{3} \over 6 l^3} k \left(1 + A_l \right)^{k}
\nonumber \\ 
& \leq & {M_{3} M_{5}^{2} \Vert B \Vert_{\infty}^{3}  \over 6 l^2}  \left(1 + A_l \right)^{k} \nonumber \\ 
& \leq & {M_{3} M_{5}^{2} \Vert B \Vert_{\infty}^{3} \over 6 l^2}  \exp(2 M_2 \Vert B \Vert_{\infty}). \nonumber
\end{eqnarray}
\end{small}
Thus, the proof is complete.  
\end{proof}

\subsection*{Proof of Lemma \ref{lemapsin}}
\begin{proof}

We will prove by induction that, for all $k \in \{1,2, \ldots ,l\}$ and  $u \in (u_{k-1}^{l} ,  u^{l}_{k}]$, we have 
\begin{equation}\label{difpsi1}
\vert \Psi^{l}(z_{1},u) -   \Psi^{l}(z_{2},u)  \vert \leq   \vert z_{1} - z_{2} \vert \prod_{j=1}^{k}  \left[ 1 + M_{2}(u^{l}_{j}-u_{j-1}^{l})  + [M_{2}^{2} + M_{3}M_{5}]{(u^{l}_{j}-u_{j-1}^{l})^{2} \over 2 }\right].
\end{equation}

Furthermore, for all $k \in \mathbb{N}$ we have obtained a global bound
\begin{equation*}
\vert \Psi^{l}(z_{1},u) - \Psi^{l}(z_{2},u) \vert \leq    \vert z_{1} - z_{2} \vert \exp(2M_{2} \Vert B \Vert_{\infty} ).  
\end{equation*}
In a similar way as in previous lemmas, if $0=u_{0}^{l} < u \leq u^{l}_{1}$, then by equation (\ref{psin}) and  the fact that $\Psi^{l}(z,u_0^l) =z$,  we have 
\begin{equation*}
\left\vert \Psi^{l}(z_{1},u) -\Psi^{l}(z_{2},u) \right\vert \leq  \left\vert  z_{1}-z_{2} \right\vert \left[ 1 + M_2 (u_{1}^{l}-u_{0}^{l})  + (M_2^2 + M_3 M_5) {(u_{1}^{l}-u_{0}^{l})^{2} \over 2} \right].
\end{equation*}
Then for $k=1$ (\ref{difpsi1}) is satisfied.
Now, consider that (\ref{difpsi1}) is true for $k$. Then, we will prove that the inequality is true for its successor,  $k+1$. For that, we will study (\ref{difpsi1}) for $u \in (u_{k}^{l} ,  u^{l}_{k+1}]$. 

Following (\ref{psin}), Lipschitz condition and hypothesis on the second derivative of $\sigma$, we have 
\begin{eqnarray*}
& &\left\vert \Psi^{l}(z_{1},u) -   \Psi^{l}(z_{2},u)   \right\vert    \\
&\leq &
\left\vert \Psi^{l}(z_1,u^{l}_{k}) - \Psi^{l}(z_2,u^{l}_{k})\right\vert + M_2 \int_{u_{k}^{l}}^{u}\left\vert \Psi^{l}\left(z_1,u^{l}_{k}\right) -\Psi^{l}\left(z_2,u^{l}_{k}\right) \right\vert ds  \\
&+& (M_2^2 + M_3 M_5) \int_{u_{k}^{l}}^{u}\left\vert \Psi^{l}\left(z_1,u^{l}_{k}\right) -\Psi^{l}\left(z_2,u^{l}_{k}\right) \right\vert (s-u^{l}_{k}) ds \\
&=& \left\vert \Psi^{l}(z_1,u^{l}_{k}) - \Psi^{l}(z_2,u^{l}_{k})\right\vert  
\left( 1+M_2 (u-u^{l}_{k}) + (M_2^2 + M_3 M_5) { (u-u^{l}_{k})^2 \over 2} \right).
\end{eqnarray*}
Consequently, from our induction hypothesis, we get
\begin{eqnarray*}
\left\vert \Psi^{l}(z_{1},u) -   \Psi^{l}(z_{2},u)   \right\vert    & \leq & \left\vert  z_{1}-z_{2} \right\vert \prod_{j=1}^{k}   \left[ 1 + M_{2}(u^{l}_{j}-u_{j-1}^{l})  + [M_{2}^{2} + M_{3}M_{5}]{(u^{l}_{j}-u_{j-1}^{l})^{2} \over 2 }\right] \\
&\times&  \left( 1 + M_{2}(u^{l}_{k+1}-u_{k}^{l})  + [M_{2}^{2} + M_{3}M_{5}]{(u^{l}_{k+1}-u_{k}^{l})^{2} \over 2 }\right)  \\
&= &  \left\vert  z_{1}-z_{2} \right\vert \prod_{j=1}^{k+1}   \left( 1 + M_{2}(u^{l}_{j}-u_{j-1}^{l})  + [M_{2}^{2} + M_{3}M_{5}]{(u^{l}_{j}-u_{j-1}^{l})^{2} \over 2 }\right),
\end{eqnarray*}
which implies that (\ref{difpsi1}) is satisfied.

Now, for all $(z,u) \in [-M,M] \times [- \Vert B \Vert_{\infty},\Vert B \Vert_{\infty}]$, there is some $k \in \lbrace 1,\ldots, l \rbrace$ such that 
$u^{l}_{k-1} < u \leq u^{l}_{k} $ and from (\ref{difpsi1})
\begin{eqnarray*}
\left\vert \Psi^{l}(z_{1},u) -   \Psi^{l}(z_{2},u)   \right\vert    & \leq &  \left\vert  z_{1}-z_{2} \right\vert \prod_{j=1}^{k}   \left( 1 + M_{2}(u^{l}_{j}-u_{j-1}^{l})  + [M_{2}^{2} + M_{3}M_{5}]{(u^{l}_{j}-u_{j-1}^{l})^{2} \over 2 }\right) \\
&\leq & \left\vert  z_{1}-z_{2} \right\vert \left[ 1 + {2M_{2} \Vert B \Vert_{\infty}  \over l} \right] ^{l} \leq \left\vert  z_{1}-z_{2} \right\vert \exp (2M_{2} \Vert B \Vert_{\infty} ),
\end{eqnarray*}
 where the last inequality is due  by the fact that for  $l$ large enough ${M_2 \Vert B \Vert_{\infty} + M_3M_5 \Vert B \Vert_{\infty} /M_2 \over 2 l} < 1$ and $\left( 1 + {2M_2 \Vert B \Vert_{\infty} \over l} \right)^l < \exp(2 M_2 \Vert B \Vert_{\infty})$. Thus the proof is complete.

\end{proof}
\subsection*{Proof of Lemma \ref{difphin1}}

\begin{proof} 
We will prove by induction that, for all $k \in \{1,\ldots, l\}$ and  $u \in (u_{k-1}^{l} ,  u^{l}_{k}]$, we have 
\begin{small}
\begin{equation}\label{difphi1}
 \vert \phi^{l}(z_{1},u) - \phi^{l}(z_{2},u)  \vert  \leq   \vert z_{1} - z_{2} \vert \prod_{j=1}^{k}  \left[ 1 + M_{2}(u^{l}_{j}-u_{j-1}^{l})  + M_{2}^{2} {(u^{l}_{j}-u_{j-1}^{l})^{2} \over 2 }\right].
\end{equation}
\end{small}
As a consequence, for all $k \in \mathbb{N}$,
\begin{equation*}
 \vert \phi^{l}(z_{1},u) - \phi^{l}(z_{2},u)  \vert 
\leq   \vert z_{1} - z_{2} \vert \exp \left( 2M_{2} \Vert B \Vert_{\infty} \right), 
\end{equation*}
is true. 

If $0=u_{0}^{l} < u \leq u^{l}_{1}$, then by equations (\ref{phi0}) and (\ref{phin1}) and  the fact that $\phi^{l}(z,u_{0}^{l}) =z$, we have 
\begin{small}
\begin{equation*}
 \vert \phi^{l}(z_{1},u) - \phi^{l}(z_{2},u)  \vert  \leq   \vert z_{1} - z_{2} \vert   \left[ 1 + M_{2}(u^{l}_{1}-u_{0}^{l})  + M_{2}^{2} {(u^{l}_{1}-u_{0}^{l})^{2} \over 2 }\right].
\end{equation*}
\end{small}
Therefore for $k=1$  (\ref{difphi1}) is  satisfied. Now, let (\ref{difphi1}) be true until $k$. Therefore, it remains  to prove that this inequality is true for its successor, $k+1$. For that, we choose   $u \in (u_{k}^{l} ,  u^{l}_{k+1}]$. 

Using (\ref{phin1}) and Lipschitz condition on $\sigma$ again, we can write  
\begin{eqnarray*}
& &\left\vert \phi^{l}(z_{1},u) -\phi^{l}(z_{2},u)   \right\vert    \\
&\leq &
\left\vert \phi^{l}(z_1,u^{l}_{k}) - \phi^{l}(z_2,u^{l}_{k})\right\vert + M_2 \int_{u_{k}^{l}}^{u}\left\vert \phi^{l}\left(z_1,u^{l}_{k}\right) -\phi^{l}\left(z_2,u^{l}_{k}\right) \right\vert ds  \\
&+& M_2^2 \int_{u_{k}^{l}}^{u}\left\vert \phi^{l}\left(z_1,u^{l}_{k}\right) -\phi^{l}\left(z_2,u^{l}_{k}\right) \right\vert (s-u^{l}_{k}) ds \\
&=& \left\vert \phi^{l}(z_1,u^{l}_{k}) - \phi^{l}(z_2,u^{l}_{k})\right\vert  
\left( 1+M_2 (u-u^{l}_{k}) + M_2^2 { (u-u^{l}_{k})^2 \over 2} \right).
\end{eqnarray*}

Our induction hypothesis leads us to 
\begin{eqnarray*}
\vert \phi^{l}(z_{1},u) - \phi^{l}(z_{2},u) \vert & \leq & \left\vert  z_{1}-z_{2} \right\vert \prod_{j=1}^{k}   \left[ 1 + M_{2}(u^{l}_{j}-u_{j-1}^{l})  + M_{2}^{2}{(u^{l}_{j}-u_{j-1}^{l})^{2} \over 2 }\right] \\
&\times&  \left[ 1 + M_{2}(u^{l}_{k+1}-u_{k}^{l})  + M_{2}^{2} {(u^{l}_{k+1}-u_{k}^{l})^{2} \over 2 }\right]  \\
&= &  \left\vert  z_{1}-z_{2} \right\vert \prod_{j=1}^{k+1}   \left[ 1 + M_{2}(u^{l}_{j}-u_{j-1}^{l})  + M_{2}^{2} {(u^{l}_{j}-u_{j-1}^{l})^{2} \over 2 }\right].
\end{eqnarray*}
Therefore, (\ref{difphi1}) for any $k \in \{ 1, \ldots , l \}$.
 
 Now, for all $u \in [- \Vert B \Vert_{\infty} ,\Vert B \Vert_{\infty}]$, there is some $k \in \lbrace 1,\ldots, l \rbrace$ such that 
$u^{l}_{k-1} < u \leq u^{l}_{k} $ and, by (\ref{difphi1}),
\begin{eqnarray*}
\left\vert \phi^{l}(z_{1},u) -   \phi^{l}(z_{2},u)   \right\vert    & \leq &  \left\vert  z_{1}-z_{2} \right\vert \prod_{j=1}^{k}   \left( 1 + M_{2}(u^{l}_{j}-u_{j-1}^{l})  + M_{2}^{2}{(u^{l}_{j}-u_{j-1}^{l})^{2} \over 2 }\right) \\
&\leq & \left\vert  z_{1}-z_{2} \right\vert \left[ 1 + {2M_{2}\Vert B \Vert_{\infty} \over l} \right] ^{l} \leq \left\vert  z_{1}-z_{2} \right\vert \exp (2M_{2}\Vert B \Vert_{\infty}),
\end{eqnarray*}
 where the last inequality is due  by the fact that for  $l$ large enough ${M_2 \over l}  < 1$ and $\left( 1 + {2M_2 \Vert B \Vert_{\infty} \over l} \right)^l < \exp(2 M_2 \Vert B \Vert_{\infty})$. Therefore (\ref{difphin1-1}) is satisfied and the proof is complete.
\end{proof}
\subsection*{Proof of Lemma \ref{lemayyn}}
\begin{proof}
By equations (\ref{y2}) and  (\ref{yl}), we have, for $t \in [0,T]$, 
\begin{small}
\begin{eqnarray*}
\left\vert Y_{t} - Y_{t}^{n} \right\vert & \leq & \int_{0}^{t} \left\vert \exp \left( -\int_{0}^{B_{s}} \sigma' \left( \phi(Y_{s} ,u) \right) du \right) b \left( \phi(Y_{s} ,B_{s})  \right) \right. \nonumber \\
& - & \left. \exp \left( -\int_{0}^{B_{s}} \sigma' ( \Psi^{n}(Y^{n}_{s} ,u) ) du \right) b \left( \Psi^{n}(Y^{n}_{s} ,B_{s})  \right) \right\vert ds \nonumber \\
& \leq & \bold{K}_{1} + \bold{K}_{2},
\end{eqnarray*}
\end{small}
with
\begin{small}
\begin{equation}
\bold{K}_{1} = \int_{0}^{t} \left\vert \exp \left( -\int_{0}^{B_{s}} \sigma' \left( \phi(Y_{s} ,u) \right) du \right) \right\vert \left\vert  b \left( \phi(Y_{s} ,B_{s})  \right) -  b \left( \Psi^{n}(Y^{n}_{s} ,B_{s})  \right) \right\vert   ds, \nonumber 
\end{equation}
\end{small}
and
\begin{small} 
\begin{equation}
\bold{K}_{2} = \int_{0}^{t} \left\vert \exp \left( -\int_{0}^{B_{s}} \sigma' \left( \phi(Y_{s} ,u) \right) du \right)  -  \exp \left( -\int_{0}^{B_{s}} \sigma' \left( \Psi^{n}(Y^{n}_{s} ,u) \right) du \right)\right\vert \left\vert b \left( \Psi^{n}(Y^{n}_{s} ,B_{s})  \right) \right\vert ds. \nonumber 
\end{equation}
\end{small}
Therefore by (\ref{eq:phi}), the Lipschitz properties on $b$ and $\sigma$, and Lemmas \ref{lemaphi1} and \ref{lemapsi1}  we obtain
\begin{small}
\begin{eqnarray}
 &&\bold{K}_{1} \leq  M_{4} \exp \left( M_2\Vert B \Vert_{\infty}  \right) \int_{0}^{t} \vert \phi(Y_{s} , B_{s}) - \Psi^{n}(Y^{n}_{s},B_{s}) \vert ds \nonumber \\
& \leq & M_{4}  \exp \left( M_2 \Vert B \Vert_{\infty}  \right) \left( \int_{0}^{t} \vert \phi(Y_{s} , B_{s}) - \phi(Y^{n}_{s},B_{s}) \vert ds 
+  \int_{0}^{t} \vert \phi(Y^{n}_{s} , B_{s}) - \phi^{n}(Y^{n}_{s},B_{s}) \vert ds \right. \nonumber \\
& + & \left.  \int_{0}^{t} \vert \phi^{n}(Y^{n}_{s} , B_{s}) - \Psi^{n}(Y^{n}_{s},B_{s}) \vert ds  \right)\nonumber \\
&\leq & M_{4} \exp \left(  M_2\Vert B \Vert_{\infty}  \right) \left(  \int_{0}^{t} \exp(M_2 \Vert B \Vert_{\infty})  \vert Y_{s} - Y^{n}_{s}|ds \right. \nonumber \\
 &&  \hspace{3cm} +\left. {T M_2^2 M_5 \Vert B \Vert^{3}_{\infty} \over 6n^2} \exp(M_2 \Vert B \Vert_{\infty})  + {T M_3 M_5^2  \Vert B \Vert^{3}_{\infty} \over 6n^2} \exp(2M_2 \Vert B \Vert_{\infty}) \right). \nonumber
\end{eqnarray}
\end{small}
Now, by the mean value theorem, we get
\begin{eqnarray*}
\bold{K}_2 & \leq & M_{1}M_{3} \exp \left( M_{2} \Vert B \Vert_{\infty}  \right)\int_{0}^{t} \int_{0}^{\vert B_{s} \vert } \vert \phi(Y_{s},u) - \Psi^{n}(Y^{n}_{s},u) \vert du ds.
\end{eqnarray*}
Hence, proceeding as in $\bold{K}_{1}$, we obtain 
\begin{small}
\begin{eqnarray*}
\bold{K}_{2} & \leq & M_{1}M_{3} \exp \left( M_{2} \Vert B \Vert_{\infty}  \right) \Vert B \Vert_{\infty} \left( \int_{0}^{t}  \exp(M_2 \Vert B \Vert _\infty ) \vert Y_{s} - Y^{n}_{s}| ds \right.\nonumber \\
& +& \left.  {T M_2^2 M_5 \Vert B \Vert_{\infty}^{3} \over 6n^2} \exp(M_2 \Vert B \Vert_{\infty})  + {TM_3 M_5^2  \Vert B \Vert_{\infty}^{3} \over 6n^2} \exp(2M_2 \Vert B \Vert_{\infty}) \right).
\end{eqnarray*}
\end{small}
Taking into account the inequalities for  $\bold{K}_{1}$ and  $\bold{K}_{2}$, we have 
\begin{eqnarray*}
\left\vert Y_{t} - Y_{t}^{n} \right\vert & \leq & C_{1} \int_{0}^{t}  \vert Y_{s}- Y^{n}_{s} \vert  ds  + {C_{2}  \over n^{2}}, 
\end{eqnarray*}
\begin{equation*}
C_{1} =  (M_{4} +  M_{1}M_{3}  \Vert B \Vert_{\infty} )  \exp(2 M_2 \Vert B \Vert_{\infty})
\end{equation*}

and 
\begin{eqnarray*}
C_{2} &=&  \exp( M_2 \vert B \Vert_{\infty})  (M_{4} +  M_{1}M_{3}\Vert B \Vert_{\infty}) \\
& \times & \left({T M_2^2 M_5 \Vert B \Vert_{\infty}^{3} \over 6} \exp(M_2 \Vert B \Vert_{\infty})  + {T M_3 M_5^2  \Vert B \Vert_{\infty}^3 \over 6} \exp(2M_2 \Vert B \Vert_{\infty}) \right)
\end{eqnarray*}
Finally, the desired result is achieved by  direct application of the Gronwall lemma.
\end{proof}

\subsection*{Proof of Lemma \ref{difynn}}

\begin{proof}
Recall that 
$h_{1}^{n}(z,u) = {\partial g^{n}(z,u) \over \partial z}$. Then, by equations (\ref{yl}) to (\ref{ynn}) we obtain that 
\begin{small}
\begin{eqnarray*}
\vert Y^{n,n}_{s} - Y^{n,n}_{t^{n}_{k}} \vert  & \leq &  \int_{t_{k}^{n}}^{s} \left\vert g^{n}\left( B_{t^{n}_{k}} , Y_{t^{n}_{k}}^{n,n} \right) + h_{1}^{n}\left(	B_{t^{n}_{k}} , Y_{t^{n}_{k}}^{n,n} 	\right) \left( B_{u} - B_{t^{n}_{k}}  \right) \right\vert du \nonumber \\
& \leq & M_{1} \exp \left( M_{2} \Vert B \Vert_{\infty} \right) (s-t_{k}^{n}) + \left\vert  h_{1}^{n}\left(	B_{t^{n}_{k}} , Y_{t^{n}_{k}}^{n,n} 	\right) \right\vert  \int_{t_{k}^{n}}^{s} \left\vert  B_{u} - B_{t^{n}_{k}}   \right\vert du \nonumber \\
& \leq & M_{1} \exp \left( M_{2} \Vert B \Vert_{\infty} \right) (s-t_{k}^{n}) + \Vert B \Vert_{H- \rho} C_{3}(s-t_{k}^{n})^{1+H-\rho} \nonumber \\
& \leq & C_{4}(s-t_{k}^{n}),
\end{eqnarray*}
\end{small}
where
\begin{small}
\begin{equation*} 
\left\vert  h_{1}^{n}\left(	B_{t^{n}_{k}} , Y_{t^{n}_{k}}^{n,n} 	\right) \right\vert  \leq    M_{1}M_{2} \exp \left(M_{2}\Vert B \Vert_{\infty} \right) +  M_{4} \exp \left(M_{2} \Vert B \Vert_{\infty} \right) M_{5}
\Vert B \Vert_{\infty}\left( 1 +M_2  \right) = C_{3},
\end{equation*}
and 
\begin{equation*}\label{C4}
C_4 =  M_{1} \exp \left( M_{2} \Vert B \Vert_{\infty} \right) + C_3 T^{H - \rho} \Vert B \Vert_{H-\rho}. 
\end{equation*}
\end{small}
The specific computation of the bound of the term $h_{1}^{n}(z,u)$ is left in the Appendix (Section \ref{apen}).
\end{proof}
\subsection*{Proof of Lemma \ref{difynynn} }
\begin{proof}
Let $n \in \mathbb{N}$ be fixed. We will prove Lemma \ref{difynynn} by induction on $k$ again. That is, for every $k \in \{1, \ldots,n \}$  and $t \in  (t_{k-1}^{n} , t_{k}^{n}]$, we have
\begin{equation}\label{dify1}
\left\vert Y^{n}_{t} - Y_{t}^{n,n} \right\vert  \leq  C_{6} \sum_{j=1}^{k} (t_{j}^n -t_{j-1}^n)^{1+2(H-\rho)} \exp(C_{7}(t^n_{k} - t^n_{j-1} )),
\end{equation}
here $0<\rho<H$. As a consequence, for all $k \in \{1, \ldots ,n\}$ we obtained the global bound  
\begin{equation*}
\left\vert Y^{n}_{t} - Y_{t}^{n,n} \right\vert  \leq  { C_{6} T^{1+2(H-\rho)}  \exp(C_{7}T) \over  n^{2(H-\rho)} },
\end{equation*}
where $C_{6}$ and $C_{7}$  are given in  (\ref{c6})  and (\ref{c7}), respectively.

First for $k=1$ and $t \in (t_{0}^{n}, t_{1}^{n}]$, equations (\ref{yl}) to (\ref{ynn}) imply  

\begin{small}
\begin{eqnarray}
\left\vert Y^{n}_{t} - Y_{t}^{n,n} \right\vert  & \leq & \int_{t_{0}^{n}}^{t} \left\vert g^{n}(B_{s}, Y^{n}_{s}) - \left[ g^{n}(B_{t_{0}^{n}},x) + h_{1}^{n}(B_{t_{0}^{n}},x)(B_{s} - B_{t_{0}^{n}})  \right] \right\vert ds \nonumber \\
& \leq & \bold{F}_{1} + \bold{F}_{2} + \bold{F}_{3}, \nonumber 
\end{eqnarray}
\end{small}
where 
\begin{eqnarray}
\bold{F}_{1} &=& \int_{t_{0}^{n}}^{t} \left\vert g^{n}(B_{s}, Y^{n}_{s}) - g^{n}(B_{s}, Y^{n,n}_{s})  \right\vert ds\nonumber \\
\bold{F}_{2} &=& \int_{t_{0}^{n}}^{t} \left\vert  g^{n}(B_{s},Y^{n,n}_{s})   -    g^{n}(B_{s},Y^{n,n}_{t_{0}^{n}})  \right\vert ds   \quad \mbox{and}\nonumber \\
\bold{F}_{3} &=& \int_{t_{0}^{n}}^{t} \left\vert g^{n}(B_{s}, Y^{n,n}_{t_{0}^{n}}) - \left[ g^{n}(B_{t_{0}^{n}},Y^{n,n}_{t_{0}^{n}}) + h_{1}^{n}(B_{t_{0}^{n}},Y^{n,n}_{t_{0}^{n}})(B_{s} - B_{t_{0}^{n}})  \right] \right\vert ds. \nonumber 
\end{eqnarray}
Equality (\ref{gl1}) and the triangle inequality allow us to write 
\begin{small}
\begin{eqnarray*}
\bold{F}_{1} &=& \int_{t_{0}^{n}}^{t} \left\vert \exp \left(-\int_{0}^{B_{s}} \sigma{'}(\Psi^{n}(Y_{s}^{n}, r)) dr \right) b(\Psi^{n}(Y_{s}^{n}, B_s)) \right.\\
 &-&  \left. \exp \left(-\int_{0}^{B_{s}} \sigma{'}(\Psi^{n}(Y_{s}^{n,n}, r)) dr \right) b(\Psi^{n}(Y_{s}^{n,n}, B_s)) \right\vert ds \\
& \leq & \int_{t_{0}^{n}}^{t} \left\vert \exp \left(-\int_{0}^{B_{s}} \sigma{'}(\Psi^{n}(Y_{s}^{n}, r))dr \right) \right\vert  \left\vert b(\Psi^{n}(Y_{s}^{n}, B_s)) - b(\Psi^{n}(Y_{s}^{n,n}, B_s)) \right\vert  ds \\
&+&  \int_{t_{0}^{n}}^{t}  \left\vert \exp \left(-\int_{0}^{B_{s}} \sigma{'}(\Psi^{n}(Y_{s}^{n}, r)) dr \right) -\exp \left(-\int_{0}^{B_{s}} \sigma{'}(\Psi^{n}(Y_{s}^{n,n}, r)) dr \right)   \right\vert \\
& \times & \left\vert b(\Psi^{n}(Y_{s}^{n,n}, B_s)) \right\vert ds.
\end{eqnarray*}
\end{small}
Therefore, the Lipschitz property on $b$ and the mean value theorem yield
\begin{small}
\begin{eqnarray*}
\bold{F}_{1} & \leq & M_{4}\exp(M_{2} \Vert B \Vert_{\infty})  \int_{t_{0}^{n}}^{t}   \left\vert \Psi^{n}(Y_{s}^{n}, B_s) - \Psi^{n}(Y_{s}^{n,n}, B_s) \right\vert  ds, \\
&+& M_{1}M_{3} \exp \left( M_{2} \Vert B \Vert_{\infty}  \right)\int_{0}^{t} \int_{0}^{\vert B_{s} \vert } \vert \Psi^{n}(Y_{s}^{n}, r) - \Psi^{n}(Y_{s}^{n,n}, r) \vert dr ds.
\end{eqnarray*}
\end{small}
Consequently, Lemma \ref{lemapsin}  lead us to
\begin{small}
\begin{eqnarray*}
\bold{F}_{1} & \leq & M_{4}\exp(3M_{2} \Vert B \Vert_{\infty})  \int_{t_{0}^{n}}^{t} \left\vert  Y^{n}_{s} - Y^{n,n}_{s}  \right\vert ds   \nonumber \\
&+& M_{1}M_3\exp(3M_{2} \Vert B \Vert_{\infty} ) \Vert B \Vert_{\infty}  \int_{t_{0}^{n}}^{t} \left\vert  Y^{n}_{s} - Y^{n,n}_{s}  \right\vert ds\\
&=& \exp(3M_{2} \Vert B \Vert_{\infty})  \left[ M_{4} + M_{1}M_3 \Vert B \Vert_{\infty} \right]  \int_{t_{0}^{n}}^{t} \left\vert  Y^{n}_{s} - Y^{n,n}_{s}  \right\vert ds.\\
\end{eqnarray*}
\end{small}
Proceeding similarly as in $\bold{F}_{1}$,

\begin{small}
\begin{eqnarray*}
\bold{F}_{2} & \leq &  M_{4}\exp(M_{2} \Vert B \Vert_{\infty})  \int_{t_{0}^{n}}^{t}   \left\vert \Psi^{n}(Y_{s}^{n,n}, B_s) - \Psi^{n}(Y^{n,n}_{t_{0}^{n}}, B_s) \right\vert  ds, \\
&+& M_{1}M_{3} \exp \left( M_{2} \Vert B \Vert_{\infty}  \right)\int_{t_{0}^{n}}^{t} \int_{0}^{\vert B_{s} \vert } \vert \Psi^{n}(Y_{s}^{n,n}, r) - \Psi^{n}(Y^{n,n}_{t_{0}^{n}}, r) \vert dr ds \\
&\leq & \exp(3M_{2} \Vert B \Vert_{\infty})  \left[ M_{4} + M_{1}M_3 \Vert B \Vert_{\infty} \right]\int_{t_{0}^{n}}^{t}  \vert Y_{s}^{n,n}-Y^{n,n}_{t_{0}^{n}} \vert  ds.
\end{eqnarray*}
\end{small}
Hence, using Lemma \ref{difynn}, we can establish 
$$\bold{F}_{2}  \leq  C_{4} \exp(3M_{2} \Vert B \Vert_{\infty})  \left[ M_{4} + M_{1}M_3 \Vert B \Vert_{\infty} \right] {(t-t_{0}^{n})^2 \over 2}.$$

Now, we deal with $\bold{F}_{3} $. From  Lemma \ref{derivadagl} (Section \ref{apen}),

\begin{eqnarray}\label{eq:F3delta}
\bold{F}_{3} & \leq & C_{5}    \int_{t_{0}^{n}}^{t}  (B_{s} - B_{t_{0}^{n}})^{2}   ds +  \int_{t_{0}^{n}}^{t}  \left\vert  \sum_{j=1}^{n}  1_{ \{B_{s} \in (u_{j-1}, u_{j}] \}} \sum_{k=1}^{j} (B_{s} - u_{k}^{n})  \Delta_{j+k}  g^n{'} \right\vert ds  \nonumber \\
& + & \int_{t_{0}^{n}}^{t}  \left\vert  \sum_{j=-n+1}^{0}  1_{ \{B_{s} \in (u_{j-1}, u_{j}] \}} \sum_{k=-j}^{0} (B_{s} - u_{k}^{n})  \Delta_{j+k}  g^n{'} \right\vert ds,
\end{eqnarray}
where 
\begin{equation*}
\Delta_{j+k} g^n{'} =  {\partial g^{n}(u_{j+k}+, Y_{t_{0}^{n}}^{n,n}) \over \partial u} - {\partial g^{n}(u_{j+k}-, Y_{t_{0}^{n}}^{n,n}) \over \partial u}
\end{equation*}
and
\begin{eqnarray*}
C_{5} &=&  \exp (M_{2} \Vert B \Vert_{\infty}) \left[   \Vert B \Vert_{\infty} (1+M_{2}) \left( M_{3} M_{1}M_{5}  + M_{2} M_{4} M_{5}  + M_{6}M_{5}  \Vert B \Vert_{\infty} (1+ M_{2})  \right)  \right. \\ 
&+& \left. M_{1}M_{2} + M_{4}M_{5}( 1+ M_{2}) \right].
\end{eqnarray*}
Note that (\ref{dergl}) implies 
\begin{eqnarray*}
\vert \Delta_{j+k} g^n{'} \vert  &\leq&  M_{4}   \exp(M_{2} \Vert B \Vert_{\infty}) \left\vert \frac{\partial \Psi^{n}(Y_{t_{0}^{n}}^{n,n}, u_{j+k}+)}{\partial u} -   \frac{\partial \Psi^{n}(Y_{t_{0}^{n}}^{n,n}, u_{j+k}-)}{\partial u} \right\vert  \nonumber \\
& = & M_{4}   \exp(M_{2} \Vert B \Vert_{\infty}) \left\vert  \sigma \left( \Psi^{n}(Y_{t_{0}^{n}}^{n,n},u_{j+k} )  \right)  -  \sigma \left( \Psi^{n}(Y_{t_{0}^{n}}^{n,n},u_{j+k-1} )  \right) \right.  \nonumber  \\
&-& \left. \sigma \sigma' \left( \Psi^{n}(Y_{t_{0}^{n}}^{n,n},u_{j+k-1} )  \right) \left(u_{j+k} - u_{j+k-1} \right) \right\vert \nonumber \\
& \le & M_{4} M_{2}  \exp(M_{2} \Vert B \Vert_{\infty})  \left[ \left\vert  \Psi^{n}(Y_{t_{0}^{n}}^{n,n},u_{j+k} ) - \Psi^{n}(Y_{t_{0}^{n}}^{n,n},u_{j+k-1} )  \right\vert  \right. \nonumber \\
& + & \left. M_{5}M_{2} (u_{j+k}-u_{j+k-1}) \right] \nonumber \\
& \le & M_{4} M_{2}  \exp(M_{2} \Vert B \Vert_{\infty})  \left[ (M_{5} + M_{5}M_{2}) {(u_{j} -u_{j-1})^{2} \over 2 } \right. \nonumber \\
& + & \left. M_{5}M_{2} (u_{j+k}-u_{j+k-1}) \right] \nonumber \\
& \le & M_{4} M_{2}  \exp(M_{2} \Vert B \Vert_{\infty})  (u_{j} -u_{j-1}) \left[ (M_{5} + M_{5}M_{2})  \Vert B \Vert_{\infty}   \right. \nonumber \\
& + & \left. M_{5}M_{2} \right] \nonumber \\
&=& C_{8} (u_{j} -u_{j-1}).
\end{eqnarray*}
Hence, (\ref{eq:F3delta}) implies  
\begin{equation*}
\bold{F}_{3}  \leq ( C_{5} +C_{8} )  \int_{t_{0}^{n}}^{t}  (B_{s} - B_{t_{0}^{n}})^{2}   ds \leq  ( C_{5} +C_{8} )  \Vert B \Vert_{H-\rho}  \frac{(t - t_{0}^{n})^{1+2(H-\rho)}}{2},
\end{equation*}
Since $H < 1/2$, then the previous estimations for $\bold{F}_{1}$, $\bold{F}_{2}$ and $\bold{F}_{3}$ give  
\begin{eqnarray*}
& &\left\vert Y^{n}_{t} - Y_{t}^{n,n} \right\vert   \leq \nonumber \\
& &   \exp(3 M_{2} \Vert B \Vert_{\infty} ) [M_{4} + M_{1} M_{3} \Vert B \Vert_{\infty}] \int_{t_{0}^{n}}^{t} \left\vert  Y^{n}_{s} - Y^{n,n}_{s}  \right\vert ds \\ 
& &+   \left[ C_{4} \exp(3 M_{2} \Vert B \Vert_{\infty} ) [M_{4} + M_{1}M_{3}\Vert B \Vert_{\infty} ]T^{1-2(H - \rho)}  \right. \\
&& \left.  + (C_{5} + C_{8}) \Vert B \Vert_{H-\rho} \right]  (t-t_{0}^{n})^{1 + 2(H-\rho)} \\
& & = C_{6}(t-t_{0}^{n})^{1+2(H-\rho)} + C_{7} \int_{t_{0}^{n}}^{t} \left\vert  Y^{n}_{s} - Y^{n,n}_{s}  \right\vert ds.
\end{eqnarray*}

Then by the Gronwall lemma and $t \in (t_{0}^{n} , t_{1}^{n}]$, we conclude
\begin{equation*}
\left\vert Y^{n}_{t} - Y_{t}^{n,n} \right\vert   \leq C_{6}  \exp(C_{7}(t^{n}_{1}-t_{0}^{n}) ) (t^{n}_{1}-t_{0}^{n})^{1+2(H-\rho)}.
\end{equation*}

Now we  show that (\ref{dify1}) is true for $k+1$ if it  holds for $k$. So we choose $t \in (t_{k}^{n}, t_{k+1}^{n}]$.  Towards this end, we proceed  as in the case $k=1$:
\begin{small}
\begin{eqnarray*}
\left\vert Y^{n}_{t} - Y_{t}^{n,n} \right\vert  & \leq & \left\vert Y^{n}_{t^n_{k}} - Y_{t^n_{k}}^{n,n} \right\vert  \\
& + & \left| \int_{t_{k}^{n}}^t \left( g^{n}\left(B_{s} , Y^{n}_{s}\right)-  \left[ g^{n}\left(	B_{t^{n}_{k}} , Y_{t^{n}_{k}}^{n,n} 	\right) + h_{1}^{n}\left(	B_{t^{n}_{k}} , Y_{t^{n}_{k}}^{n,n} 	\right) \left( B_{s} - B_{t^{n}_{k}}  \right) \right] \right) ds \right| \\
& \leq &  C_{6} \sum_{j=1}^{k} (t_{j}^n -t_{j-1}^n)^{1+2(H-\rho)} \exp(C_{7}(t^n_{k} - t^n_{j-1} )) + C_{6}(t_{k+1}^n - t_{k}^n )^{1 + 2(H-\rho)}  \\
&+& C_{7} \int_{t_{k}^{n}}^{t} \left\vert  Y^{n}_{s} - Y^{n,n}_{s}  \right\vert ds. 
\end{eqnarray*}
\end{small}
Therefore, using the Gronwall lemma again and $t \in (t_{k}^{n} , t_{k+1}^{n}]$
\begin{small}
\begin{eqnarray*}
& & \left\vert Y^{n}_{t} - Y_{t}^{n,n} \right\vert  \leq \\
 & \leq & \left[ C_{6} \sum_{j=1}^{k} (t_{j}^n -t_{j-1}^n)^{1+2(H-\rho)} \exp(C_{7}(t^n_{k} - t^n_{j-1} )) + C_{6}(t_{k+1}^n - t_{k}^n )^{1 + 2(H-\rho)} \right] \exp(C_{7}(t_{k+1}^n - t_{k}^n ))   \\ 
&=&   C_{6} \sum_{j=1}^{k} (t_{j}^n -t_{j-1}^n)^{1+2(H-\rho)} \exp(C_{7}(t^n_{k+1} - t^n_{j-1} )) + C_{6}(t_{k+1}^n - t_{k}^n )^{1 + 2(H-\rho)} \exp(C_{7}(t_{k+1}^n - t_{k}^n )) \\
&=&  C_{6} \sum_{j=1}^{k+1} (t_{j}^n -t_{j-1}^n)^{1+2(H-\rho)} \exp(C_{7}(t^n_{k+1} - t^n_{j-1} )). 
\end{eqnarray*}
\end{small}
Therefore (\ref{dify1}) is true for any $k \leq n$.

Finally, for all $t \in [0,T]$, there exits $k \in \lbrace 1,\ldots, n \rbrace$ such that 
$t^{n}_{k-1} < t \leq t^{n}_{k} $. Thus (\ref{dify1}) implies 
\begin{eqnarray*}
\left\vert Y^{n}_{t} - Y_{t}^{n,n} \right\vert  & \leq & C_{6} \left({T \over n} \right)^{1+2(H-\rho)} k \exp(C_{7}(t^n_{k} - t^n_{0} )) \\
& \leq & C_{6} T \left({T \over n} \right)^{2(H-\rho)}  \exp(C_{7}T),
\end{eqnarray*}
and the proof is complete.
\end{proof}
\section{Appendix}\label{apen}
Here, we consider the following useful result for the analysis of the  convergence of the scheme.

\begin{lem}\label{derivadagl}
Let $\left\lbrace  u_{i}^{l} \right\rbrace$ be a partion of the interval $[-R,R]$ given by $-R=u^{l}_{-l} < \ldots <u^{l}_{-1}< u^{l}_{0}=0< u^{l}_{1}< \ldots < u^{l}_{l} = R$ and  $f:[-R,R] \rightarrow \mathbb{R} $ a $C^{2}([u_{j}^l,u_{j+1}^l])$-function  for each $j \in \{-l, \ldots, l-1 \}$. Also let $f$ be continuous on $[-R,R]$, $C$  a constant such that 
\begin{equation*}
\sup\limits_{j \in \{-l,...,l-1 \}} \Vert f'' \Vert_{\infty , [u_{j}, u_{j+1} ]} =C< \infty,
\end{equation*}
and  $x \in (u_{j}^{l} ,  u^{l}_{j+1}]$ and $y \in (u_{j+k}^{l} ,  u^{l}_{j+k+1}]$. Then, 
\begin{equation}\label{conj}
\vert f(y)-f(x)-f'(x+)(y-x) \vert  \leq {C \over 2} (y-x)^2 + \sum_{p=1}^{k} \Delta_{j+p}f' \cdot (y-u_{j+p}),
\end{equation}
where 
\begin{equation*}
\Delta_{j+p}f' =  \vert f'(u_{j+p}+)-f'(u_{j+p}-) \vert. 
\end{equation*}
\end{lem}
\begin{proof}
We will prove that (\ref{conj}) holds via induction on $k$.  

 We start our induction with $k=1$. That is, we consider two consecutive intervals.  If $x \in (u_{j},u_{j+1}]$ and $y \in (u_{j+1},u_{j+2}]$. Then, 

\begin{eqnarray*}\lefteqn{
\vert f(y)-f(x)-f'(x+)(y-x) \vert} \\
& \leq & \vert f(y)-f(u_{j+1}) - f'(u_{j+1}+)(y-u_{j+1})\vert \\
&+& \vert f(u_{j+1}) + f'(u_{j+1}+)(y-u_{j+1}) - f(x) - f'(x+)(y-x) \vert \\
& \leq & {C \over 2} (y- u_{j+1})^2 + \vert f(u_{j+1})-f(x) - f'(x+)(u_{j+1}-x)\vert \\
&+& \vert f'(u_{j+1}+)-f'(x+) \vert (y-u_{j+1}) \\
& \leq & {C \over 2} (y- u_{j+1})^2 + {C \over 2} (u_{j+1}-x)^2 \\
&+& \left[ \vert f'(u_{j+1}+)-f'(u_{j+1}-)  \vert + \vert f'(u_{j+1}-)-f'(x+)  \vert  \right] (y-u_{j+1})\\
&=& {C \over 2} (y- u_{j+1})^2 + {C \over 2} (u_{j+1}-x)^2  + (y-u_{j+1}) (\Delta_{j+1}f') \\
&+& C(u_{j+1} -x )(y-u_{j+1})\\
&=& {C \over 2} (y- x)^2  + (\Delta_{j+1}f')(y-u_{j+1}).
\end{eqnarray*}
It means, (\ref{conj}) holds for $k=1$. 

It remains  to prove that the inequality (\ref{conj}) is true for its successor, $k+1$ assuming that until $k$ is satisfied. To do so, choose
 $x \in [u_{j},u_{j+1}]$ and $y \in [u_{j+k+1},u_{j+k+2}]$. Then,
\begin{eqnarray*}\lefteqn{
\vert f(y)-f(x)-f'(x+)(y-x) \vert } \\
& \leq &  \vert f(y)-f(u_{j+k+1}) - f'(u_{j+k+1}+)(y-u_{j+k+1})\vert \\
&+& \vert f(u_{j+k+1}) + f'(u_{j+k+1}+)(y-u_{j+k+1}) - f(x) - f'(x+)(y-x) \vert \\
& \leq & {C \over 2} (y- u_{j+k+1})^2 + \vert f(u_{j+k+1})-f(x) - f'(x+)(u_{j+k+1}-x)\vert \\
&+& \vert f'(u_{j+k+1}+)-f'(x+) \vert (y-u_{j+k+1}).
\end{eqnarray*}
Hence, our induction hypothesis implies 
\begin{eqnarray*}\lefteqn{
\vert f(y)-f(x)-f'(x+)(x-y) \vert} \\
& \leq & {C \over 2} (y- u_{j+k+1})^2 + {C \over 2} (u_{j+k+1}-x)^2 +  \sum_{p=1}^{k} (u_{j+k+1} -u_{j+p})  \Delta_{j+p}f'  \\
& + & \left[  \Delta_{j+k+1}f' + C (u_{j+k+1} - u_{j+k}) +  \left\vert f'(u_{j+k}+) - f'(x+) \right\vert  \right] (y-u_{j+k+1}) \\
& \leq & {C \over 2} (y- u_{j+k+1})^2 + {C \over 2} (u_{j+k+1}-x)^2 +  \sum_{p=1}^{k} (u_{j+k+1} -u_{j+p})  \Delta_{j+p}f'  \\
&+&  \left[  \Delta_{j+k+1}f' + C (u_{j+k+1} - u_{j+k}) \phantom{  \sum_{p=1}^{k} }  \right.  \\
&+& \left.  \sum_{p=1}^{k}  \Delta_{j+p} f'  + C(u_{j+k} - x) \right] (y-u_{j+k+1}) \\
& \leq &  \frac{C}{2}(y-x)^{2} +  \sum_{p=1}^{k} (y -u_{j+p})  \Delta_{j+p}f' + (y-u_{j+k+1}) \Delta_{j+k+1}f'.  
\end{eqnarray*}
Therefore, (\ref{conj}) is satisfied for $k+1$ and the proof is complete.
\end{proof}

{\bf Acknowledgments} 
We would like to thank anonymous referee for useful comments and suggestions that improved the presentation of the paper.
Part of this work was done while Jorge A. Le\'on was visiting CIMFAV, Chile,
 and H\'ector Araya and Soledad Torres were visiting Cinvestav-IPN, Mexico. 
The authors thank both Institutions for  their hospitality and economical support. Jorge A. Le\'on was partially supported by the CONACYT grant 220303. Soledad Torres was partially supported by Proyecto ECOS C15E05; Fondecyt 1171335, REDES 150038 and  Mathamsud 16MATH03. H\'ector Araya was partially supported by  Beca 
CONICYT-PCHA/Doctorado Nacional/2016-21160138; Proyecto ECOS C15E05,  REDES 150038 and  Mathamsud 16MATH03.


\end{document}